\documentclass[a4paper,11pt,twoside,reqno]{amsart} 
\usepackage{geometry}
\usepackage[dvipsnames]{xcolor}

 \usepackage[applemac]{inputenc}
 
\usepackage{multicol} 

\usepackage{mathrsfs}
\usepackage{mathabx}

\usepackage[all]{xy}

\usepackage{hyperref}
\usepackage{graphicx}
\usepackage{amssymb}

\usepackage{enumerate}

\usepackage{tikz}

\newtheorem{theorem}{Theorem}

\newtheorem{proposition}{Proposition}[section]
\newtheorem{lemma}[proposition]{Lemma}
\newtheorem{corollaire}[proposition]{Corollary}

\newtheorem{definition}[proposition]{Definition}
\newtheorem{example}[proposition]{Example}

\newtheorem{remarque}[proposition]{Remark}

\newcommand{\bl}{\begin{lemma}}
\newcommand{\bp}{\begin{proposition}}
\newcommand{\bt}{\begin{theorem}}
\newcommand{\bc}{\begin{corollaire}}
\newcommand{\be}{\begin{equation}}
\newcommand{\bee}{\begin{equation*}}
\newcommand{\bd}{\begin{definition}}
\newcommand{\bdp}{\begin{definitionproposition}}
\newcommand{\bex}{\begin{example}}
\newcommand{\br}{\begin{remarque}}
\newcommand{\bpr}{\begin{proof}}

\newcommand{\el}{\end{lemma}}
\newcommand{\ep}{\end{proposition}}
\newcommand{\et}{\end{theorem}}
\newcommand{\ec}{\end{corollaire}}
\newcommand{\ee}{\end{equation}}
\newcommand{\eee}{\end{equation*}}
\newcommand{\ed}{\end{definition}}
\newcommand{\edp}{\end{definitionproposition}}
\newcommand{\eex}{\end{example}}
\newcommand{\er}{\end{remarque}}
\newcommand{\epr}{\end{proof}}

\newcommand{\secref}[1]{Section~\ref{#1}}

\newcommand{\thmref}[1]{Theorem~\ref{#1}}
\newcommand{\propref}[1]{Proposition~\ref{#1}}
\newcommand{\lemref}[1]{Lemma~\ref{#1}}
\newcommand{\corref}[1]{Corollary~\ref{#1}}

\newcommand{\remref}[1]{Remark~\ref{#1}}
\newcommand{\exemref}[1]{Example~\ref{#1}}

\newcommand{\parref}[1]{Paragraph~\ref{#1}}

\def\ov{\overline}

\def\menos{\backslash}

\renewcommand{\int}[1]{{\rm int} (#1)}

\def\1{{\boldsymbol 1}}
\def\gd{{\mathfrak{d}}}
\def\gC{{\mathfrak{C}}}
\def\gH{{\mathfrak{H}}}
\def\tc{{\mathtt c}}
\def\tv{{\mathtt v}}
\def\tw{{\mathtt w}}
\def\tu{{\mathtt u}}
\def\tn{{\mathtt n}}
\def\ts{{\mathtt s}}

\def\tx{{\mathtt x}}
\def\ty{{\mathtt y}}
\def\tN{{\widetilde{N}}}
\def\crH{{\mathscr H}}

\def\EE{\mathcal E}

\def\T{\mathcal G}
\def\H{\mathcal H}
\def\T{\mathcal T}

\def\I{\mathcal I}

\def\SS{\mathcal S}
\def\SSsing{\mathcal S^{sing}}
\def\SSreg{\mathcal S^{reg}}
\def\TT{\mathcal T}

\def\TTreg{\mathcal T^{reg}}
\def\RR{\mathcal R}
\def\N {\mathbb N}

\def\R {\mathbb R}

\def\Z {\mathbb Z}
\renewcommand\S {\mathbb S}
\renewcommand\L {\mathcal L}
\def\ker{{\rm Ker\,}}
\def\Id{{\rm Id\,}}
\def\coker{{\rm Coker\,}}
\def\im{{\rm Im\,}}
\def\id{{\rm Ide\,}}

\def\Hom{{\rm Hom}}
\def\Ext{{\rm Ext}}

\def\codim{{\rm codim\,}}
\def\pr{{\rm pr}}

\def\depth{{\rm depth \,}}

\def\rc{{\mathring{\tc}}}
\def\reg{{\rm reg}}
\def\sing{{\rm sing}}

 \newcommand\rojo {\color{black}}
  
    \newcommand\negro {\color{black}}

\newcommand{\bi}[2]{{#1}^{^{#2}}}
\newcommand{\dos}[2]{{#1}_{_{#2}}}
\newcommand{\tres}[3]{{#1}^{^{ #2 }}_{_{#3}}}

\newcommand{\Hiru}[3]{{#1}^{^{#2}}{( #3 )}}
\newcommand{\hiru}[3]{{#1}_{_{#2}}{( #3 )}}
\newcommand{\lau}[4]{{#1}^{^{#2}}_{_{#3}}{\left( #4 \right)}}

\newcommand{\IH}{\mathscr H}

\title[Refinement invariance]{Refinement invariance of  intersection (co)homologies}
\date{\today}

\author{Martintxo Saralegi-Aranguren}
\address{      Universit\'e d'Artois\\
      UR 2462 \\
Laboratoire de Math{\'e}matiques de Lens (LML)\\ 
F-62300 Lens, France}
\email{martin.saraleguiaranguren@univ-artois.fr}

\thanks{The author would like to thank Daniel Tanré for many helpful suggestions and comments in the preparation of this work}

\makeindex   

\begin{document} 

\begin{abstract} 
We prove the refinement invariance of several intersection (co)homo\-logies existing  in the literature:  Borel-Moore, Blown-up, the classical one, $\ldots$
These (co)homo\-logies have been introduced in order to establish the Poincaré  Duality in various contexts.
In particular, we retrieve the classical topological invariance of the intersection homology as well as several refinement invariance results already known.
\end{abstract}

\maketitle

Let us consider a topological space $X$ supporting two stratifications $\SS$ and $\mathcal T$.
We say that $(X, \SS)$ is a \emph{refinement} of  $(X, \mathcal T)$ if each stratum of $\mathcal T$ is a union of strata of $\SS$.
In this work we answer  the following question about the invariance property of the intersection homology:

\begin{center}
\em 
Can we find two perversities $\ov p
$ and $\ov q$ 
such that 

the identity $I \colon X \to X$ induces the isomorphism 
\begin{equation}\label{eq:Gol}
\lau  H {\ov p} * {X, \SS} \cong  \lau H {\ov q} *{X, \mathcal T}?
\end{equation}
\end{center}

\medskip

For pseudomanifolds and using the original  \emph{Goresky-MacPherson} perversities, an answer comes directly from the topological invariance of the intersection homology  \cite[Corollary pag. 148]{MR572580} (see also \cite[Theorem 9]{MR800845}): it suffices to take $\ov p = \ov q$. In other words, the intersection homology does not depend on the chosen stratification.
We  work in a more general setting.

 \smallskip 
 
 $\bullet${ \bf Spaces}. We do not work with pseudomanifolds, but with the more general notion of 
 CS-set (cf. \secref{sec:CS}).  They are 
locally cone-like spaces, but their links are not necessarily  pseudomanifolds.

\medskip

$\bullet${ \bf Perversities} (cf. \parref{subsec:perv}). We deal with the more  general notion of perversity introduced by MacPherson in \cite{RobertSF}: the \emph{M-perversities}. This kind of  perversity $\ov p$ associates a number $\ov p (S) \in \ov \Z = \Z \sqcup\{-\infty,\infty\}$ to any stratum $S$ of the CS-set, while a \emph{classical perversity} $\ov p$ associates a number $\ov p(\codim S)$ to the codimension of the stratum.
An $M$-perversity strongly depends on the stratification and so the topological invariance of the related intersection homology does not apply.

\smallskip

$\bullet${ \bf (Co)homologies} (cf. \secref{homcoho}). 
We consider not only the intersection homology $\tres H  {\ov p}  *$, but also the following:

\indent  {\bf + Intersection cohomologies  $\tres H * {\ov p}, \tres H {*} {\ov p,c}    $.}  The intersection cohomology $\tres H * {\ov p}$ is the cohomology of the complex defined by using the functor $\Hom$ over intersection chains.
The version with compact supports is $\tres H {*} {\ov p,c}$.

\indent  {\bf  + Tame intersection homology $\tres \gH  {\ov p}*$.} A variation of the intersection homology avoiding intersection chains who live in singular strata.
We have $\tres \gH  {\ov p}* = \tres H  {\ov p} *$ when $\ov p$ is smaller than the  top perversity $  \ov t$.
This homology is isomorphic to the blown-up intersection cohomology with compact supports $\tres \IH * {\ov p,c}$ through the Poincar\'e duality \cite{CST7} when one works with pseudomanifolds.

\indent  {\bf + Tame intersection cohomologies $\tres \gH *{\ov p} , \tres \gH *{\ov p,c}  $.}   
 The tame  intersection cohomology $\tres \gH *{\ov p}$  is the cohomology of the complex defined  by using the functor $\Hom$ over  tame intersection chains.
  The version with compact supports is $\tres \gH {*} {\ov p,c}$.
   We have $\tres \gH  * {\ov p} = \tres H * {\ov p}$ and $ \tres \gH  * {\ov p,c} = \tres H * {\ov p,c}$ when $\ov p \leq \ov t$.
The cohomology $\tres \gH *{\ov p,c}  $ is dual to $\tres \gH {D\ov p} {n-*}$ through the Poincar\'e duality \cite{MR3046315,LibroGreg} when the coefficient is a field\footnote{ In fact, following \cite{MR699009}, \emph{locally $\ov p$-torsion free.} } and one works with pseudomanifolds.
  The cohomology  $\tres \gH *{\ov p}$ is isomorphic to the de Rham intersection cohomology through the de Rham duality  \cite{MR2210257}. Here, we use real coefficients.

\indent  {\bf + Borel-Moore intersection homology $\tres H {BM,\ov p} * $.}
Similar to the intersection homology using locally finite chains instead of finite chains.

\indent  {\bf  + Borel-Moore tame intersection homology $\tres \gH {BM,\ov p} * $.}
Similar to $\tres \gH *{\ov p}$ using locally finite chains instead of finite chains.
 We have $\tres \gH  {BM,\ov p}* = \tres H  {BM,\ov p} *$ when $\ov p \leq \ov t$.

\indent  {\bf + Blown-up intersection cohomologies $\tres \IH *{\ov p}, \tres \IH *{\ov p,c}  $.}
These cohomologies are defined by using simplicial cochains, inspired by Dennis Sullivan's approach to rational homotopy theory.
The  compact supports version $\tres \IH *{\ov p,c}$ (resp. closed supports version $\tres \IH *{\ov p}$ )  is isomorphic to the tame intersection homology $\tres \gH {\ov p} {n-*}$ (resp. Borel-Moore  tame intersection homology $\tres \gH {BM,\ov p} {n-*}$)  through the Poincar\'e duality  \cite{CST5} (resp. cf. \cite{ST1})
when one works with pseudomanifolds.
We have $\tres \IH  * {\ov p} = \tres \gH * {D\ov p}$ and $\tres \IH  * {\ov p,c} = \tres \gH * {D\ov p,c}$ when the coefficient ring is a field$^1$ (cf. \cite[Theorem F]{CST5} and \cite[Corollary 13.1]{CST7}). 

\bigskip

In this paper, we give two answers to  the question \eqref{eq:Gol}.

\smallskip

$-$  \emph{Pull-back}. Given a perversity $\ov q$ on $(X,\mathcal T)$  we take $\ov p$ the pull-back perversity $I^\star \ov q$ on $(X,\SS)$ (cf. \parref{subsec:perv}). We prove
$$
\lau  H {I^\star \ov q} * {X, \SS} \cong \lau H {\ov q} *{X, \mathcal T},
$$
and similarly for the other (co)homologies used in this work (cf. \thmref{thm:refi}).

\smallskip

  $-$ \emph{Push-forward}. Given a perversity $\ov p$ on $(X,\mathcal S)$
  we take $\ov q$ the push-forward perversity $I_\star \ov p$ on $(X,\TT)$ (cf. \parref{subsec:perv}). We prove
  $$
\lau  H {\ov p} * {X, \SS} \cong \lau H {I_\star\ov p} *{X, \mathcal T},
$$
and similarly for the other (co)homologies used in this work (cf. \thmref{thm:Coars}).
In this case, we need the following conditions on $\ov p$:

\medskip

(K1) \hfill 
 $
 \ov p(Q) \leq \ov p(S) \leq \ov p(Q) + \ov t(S) - \ov t(Q),
 $
 \hfill
 \phantom{----------}
 
\medskip
 
 \noindent for any strata $S,Q \in \SS$ with $S\subset \ov Q$ and $S,Q \subset T$ for some stratum $T \in \TT$,
and

\medskip

(K2) 
\hfill 
$
 \ov p(Q) = \ov p(S) ,
$ 
\hfill
 \phantom{---------------------------------------}
 
 \medskip
 
\noindent for any strata $S,Q \in \SS$ with $\dim S =\dim Q$ and  $S,Q \subset T$ for some stratum 
$T \in \TT$\footnote{	\remref{rem:class} gives a relation between classical perversities and those verifying (K1), (K2).}.

These results encompass some other already known results about  invariance of intersection (co)homology:
\begin{itemize}
\item[-] The topological invariance of $ \tres H  {\ov p}*$  for pseudomanifolds \cite{MR572580} and  CS-sets \cite{MR800845,2019arXiv190406456F}.
\item[-] The topological invariance of  $ \tres  \gH  {\ov p}*$ and  $ \tres  \gH  {BM,\ov p}*$   for pseudomanifolds \cite{MR2209151,MR2507117}.

\item[-] The topological invariance of $ \tres  \crH * {\ov p}$ and  $ \tres  \crH * {\ov p,c}$  for CS-sets \cite{CST5,CST7}.
\item[-]  The refinement invariance of $ \tres H  {\ov p}*$  for PL-pseudomanifolds \cite{MR3175250} using $M$-perversities.

\item[-] The refinement invariance of  $ \tres H  {\ov p}*$ and 
$ \tres \gH  {\ov p}*$ for CS-sets \cite{CST3}  using $M$-perversities.

\end{itemize}

\medskip

Recently, the topological/refinement  invariance of the intersection homology has been extended to the more general setting of the torsion sensitive intersection homology (cf. \cite{2019arXiv190707790F,2019arXiv190406456F}).

\bigskip

We end this introduction by giving 
an idea of the proof of Theorems  \ref{thm:Coars} and \ref{thm:refi}.  The original proof of the classical topological invariance of the intersection homology given by King in  \cite{MR800845}  uses the intrinsic stratification $\SS^*$. He proves that the identity map $I\colon X \to X$ induces an isomorphism between the intersection homology of $(X,\mathcal S)$ and that of $(X,\SS^*)$. This gives the topological invariance since $\SS^* = \TT^*$. 

\rojo The  proof uses the Mayer-Vietoris technique in order to reduce the question to a local one. Near a point $x$ of $X$ the identity $I \colon (X,\SS) \to (X,\SS^*)$ essentially becomes a stratified map 
\begin{equation}\label{WL}
h \colon  \R^m \times \rc W 
\to \R^k \times \rc L
\end{equation}
(see for exemple \cite[Section 5.5]{LibroGreg}). Here,   $W$ is the link of $x$ relatively to $\SS$ and $L$ denotes the link  of $x$ relatively to $\SS^*$. Working with the usual local calculations of intersection homology one proves that $h$ is a quasi-isomorphism for the intersection homology. \negro

In our context, we don't know wether the identity map $I \colon (X,\SS) \to (X,\TT)$ has the nice local description  \eqref{WL}. We proceed in a different way. We construct a finite sequence of CS-sets
\begin{equation}\label{sucesion}
(X,\SS)  = (X,\mathcal R_0)\stackrel I \longrightarrow  (X,\mathcal R_1)
\stackrel I \longrightarrow  \cdots
(X,\mathcal R_{\ell-1} ) \stackrel I \longrightarrow   (X,\mathcal R_\ell) =(X,\mathcal T),
\end{equation}
where each step \rojo $I \colon (X,\mathcal R_{j-1}) \to \mathcal (X,R_j)$,  called	 {\em simple refinement},  is a refinement having the  local description 
\begin{equation}\label{WLBis}
h \colon  \R^m \times \rc (\S^{b-1} * L)
\to \R^k \times \rc L,
\end{equation}
 where $\S^{b-1}$ stands for a $(b-1)$-homological sphere
 (see \propref{Max} for a complete picture). In the simple refinement  $I \colon (X,\mathcal R_{j-1}) \to \mathcal (X,\mathcal R_j)$ the strata of $\mathcal R_{j-1}$ are the same as those of $\mathcal R_{j}$ except the following ones:
 \begin{equation}\label{desapa}
 \{ \R^k \times \{\tw\}  , \R^k \times \S^{b-1} \times ]0,1[\}, \hbox{ dans } \mathcal R_{j-1},
 \hbox{ becoming } \R^{m} \times \{ \tv\},  \hbox{ dans } \mathcal R_{j}.
  \end{equation}
   Here, $\tw$ (resp. $\tv$) is the apex of the cone $\rc W$ (resp. $\rc L$).
 
\negro
Now, we can follow the procedure of \cite{MR800845} in order to get the isomorphism between the intersection homologies of $(X,\SS)$ and $(X,\TT)$.

The construction of this sequence uses the fact  that  any stratum of $S \in\SS$ is included in a stratum $T \in \TT$. This gives the following dichotomy: $S$ is a {\em source stratum}  if $\dim S = \dim T$ and $S$ is a  {\em virtual stratum} if  $\dim S < \dim T$. Somehow, the virtual strata of $\SS$ disappear in $\TT$ while source strata become larger. The first step 			in the construction of the above sequence is to {\em eliminate} the maximal virtual strata of $\SS$. In this way, we obtain in the CS-set $(X,\mathcal R_1)$.
We continue  applying  this principle to the refinement $(X,\mathcal R_1)
\stackrel I \longrightarrow (X,\TT)$ and we eventually get \eqref{sucesion}.

What  do we mean by ``eliminate"? Let us suppose that $S \in \SS$ is the unique virtual stratum. There exist two source strata  $R_0,R_1 \in \SS$ with $T = S \cup R_0 \cup R_1$  (maybe $R_0=R_1$). We replace the strata  $S,R_0,R_1$ of $\SS$ by the stratum $T$ in order to get $\RR_1$, that is, 
$
\mathcal R_1 = \{ Q \in \SS \mid   Q \ne S,R_0,R_1\} \cup \{ T\}
$
  (cf. \rojo \eqref{desapa} \negro or \exemref{ejemplo} for a richer situation).

\medskip

Exceptional strata are singular strata of $\SS$ included in regular strata1 of $\T$.
For example, the apex of the open cone of the sphere $S^0$ (which is indeed an open interval) is an exceptional stratum if we take 
$\T$ the one-stratum stratification of that interval. 
This example is a limit case
 for the refinement invariance results we establish in this work (see \remref{rem:cod1excep}).

\bigskip

The embedding condition asked in \cite{SRef1} gives  that the local description of the refinement 
$I \colon (X,\mathcal R_{j-1}) \to \mathcal (X,\mathcal R_j)$ is  similar, but stronger,   than  \eqref{WLBis}. In fact, it replaces  the homological sphere $\S^{b-1}$ by  the true sphere $S^{b-1}$
(cf. \cite[Proposition 4.4 (c)]{SRef1}.

\bigskip

\noindent {\bf Nota bene.}  There is a mistake in \cite[Paragraph 5.3]{SRef1} which rends incorrect the proof the Corollary 5.10, op. cit. In fact, the author does not know if the intrinsic stratification verifies the strata embedding condition necessary to apply Theorem A, op. cit. 

Nevertheless, the intrinsic stratification is a refinement in the context of this paper and the proof of \corref{Impl2} is correct. In other words, topological invariance can be deduced from refinement invariance.
	
\bigskip

For a topological space $X$, we denote by $\tc X= X \times [0,1]/ (X \times \{ 0\})$ the \emph{cone} on $X$ and $\rc X = X \times [0,1[/ (X \times \{ 0\})$ the \emph{open cone} on $X$. A point of the cone is denoted by $[x,t]$. The apex of the cone is $\tv =[-,0]$.

We shall write
\rojo $S^m$ the $m$-dimensional sphere and $\S^m$ an $m$-dimensional homological sphere,
 %$B^m = \{ z \in \R^m \mid ||z||<1\}$ and $D^m = \{ z \in \R^m \mid ||z||\leq 1\}$, 
$m\in \N$.
\negro

\bigskip

\tableofcontents

\section{Intersection homologies and cohomologies (filtered spaces)}\label{homcoho}

\begin{quote}
We present the homologies and cohomologies   studied in this work.
We review their main computational properties which we are going to use in the proof of Theorems \ref{thm:Coars} and \ref{thm:refi}.
\end{quote}

\subsection{Filtered spaces}

A \emph{filtered space} is a  Hausdorff topological space endowed with a filtration by closed sub-spaces 
\bee
\emptyset = X_{-1} \subseteq X_0\subseteq X_1\subseteq\ldots \subseteq X_{n-1} \subsetneq X_n=X.
\eee
The
\emph{formal dimension} of $X$ is $\dim X=n$. Any non-empty connected component $S$ of a $X_i\menos X_{i-1}$ is a \emph{stratum}. We say that $i$ is the \emph{formal dimension} of $S$, written $i = \dim S$. We denote by $\SS$ the family of strata. In order to avoid confusion we also write $(X,\SS)$ the filtered space.
The $n$-dimensional strata are the \emph{regular strata}, other strata are \emph{singular strata}. The family of singular strata  is denoted by $\SSsing$.
Their union is the singular part  $\Sigma$.

% A filtered space $(X,\SS)$ is a \emph{strong filtered space} if the following frontier condition is verified:
%$$  S,S' \in \SS \hbox{ with } S\cap \overline{S'}\neq \emptyset \Longrightarrow S\subset \overline{S'}.
%$$
%The relation $S \preceq S'$, defined on the set of strata by $S\subset \overline{S'}$, is an order relation on $\SS_\F$ (see \cite[Proposition A.22]{CST1}). The notation $S \prec S'$ means $S \preceq S'$ and $S \ne S'$.

A continuous map $f \colon (X,\SS) \to (Y,\T)$  between two filtered spaces is a \emph{stratified map} if for each  $S \in \SS$ there exists $\bi S f \in \T$ with $f(S) \subset \bi Sf$ and $\codim \bi Sf \leq \codim S$.
The map $f$ is a \emph{stratified homeomorphism} if $f$ is a homeomorphism and $f^{-1}$ is a stratified map.

\subsection{Examples} \label{Ejemplos} Unless expressly stated otherwise, a manifold $M$ is endowed with the filtration $\emptyset  \subseteq M$. The associated filtration is denoted by $\I =\{M_{cc} \}$,  where $cc$ denotes connected component.
Consider $(X,\SS)$  a filtered space. 

\indent + An open subset $U \subset X$ inherits the \emph{induced  filtration }$U_i = U \cap X_i$. The 	associated \emph{induced stratification}  is $\SS_U = \{ (S \cap U)_{cc} \mid S \in \mathcal S\}$.
 We write $(U,\SS)$ instead of $(U,\SS_{U})$.

\indent + Given an $m$-dimensional manifold $M$ the product $M \times X$ inherits the \emph{product filtration} $(M\times X)_i = M \times X_{i-m}$. The associated \emph{product stratification}  is $\I \times \SS = \{ M_{cc} \times S \mid S\in \SS\}$.

\indent + The cone $\rc X$ inherits the \emph{cone filtration} $\rc X_i = \rc X_{i-1}$, with the convention $\rc \emptyset =\{\tv\}$. The 
associated \emph{cone stratification} is $\rc \SS = \{ \{\tv\} \} \sqcup \{ S \times ]0,1[, S \in \SS\}$.

\rojo
\indent +  Let $m\in \N$. We consider the join $\S^{m} * X = \tc\S^{m}\times X/\sim$, where the equivalence relation is generated by 
 $
 ([z,1],x) \sim ([z,1],x').
 $
 An element of \, $\S^{m} * X$ is denoted by $[[z,t],x]$. We identify $\S^{m}$ with 
$\{ [[z,1],x]   \}$ and $X$ with $\{[[z,0],x]\}$. The join $\S^{m} * X$ is  endowed with the \emph{join filtration}
$
\S^{m} \subset \S^{m} *X_0 \cdots \subset \S^{m} * X_{n-1} \subsetneq \S^{m} * X_{n}.
$
The associated \emph{join stratification} is 
$
\SS_{\star m } =\{(\S^{m})_{cc}, \rc \S^{m} \times  S \mid S \in \SS\}$.
%
%The case $m=0$ is in fact the suspension $\Sigma X = S^0 * X$. If $S^0 = \{ \tn, \ts\}$ then $\Sigma \SS = \SS_0 = \{ \{\tn\}, \{\ts\},]-1,1[ \times  S \ \mid S \in \SS\}$.

\indent + Any point  $x \in X$, with $\{x \} \not\in \SS$,  can be added as a new stratum. The new filtration is 
$X_0 \cup \{ x \} \subset X_1\cup \{ x \}   \subset \cdots \subset  X_n \cup \{ x \} =X$ and the associated stratification is
\begin{equation}\label{ptplus} \SS_x =  \{ x \} \sqcup \{ (S\menos \{x\})_{cc}\mid S \in \SS\}.
\end{equation}
\negro

The following result will be important for the understanding of the local structure of  the filtered spaces we are interested in.

\rojo
\bp\label{join1}
Let $(X,\SS)$ be a filtered space and let  $m\in \N^*$. Then, there exists a stratified homeomorphism  
\be
\label{refi}
h \colon (\rc (\S^{m} * X), \rc\SS_{\star m}) \to (\rc \S^m \times \rc X, (\I \times \rc \SS)_{(\tv,\tw)}),
\ee
where $\tv$ is the apex of $\rc \S^m$ and $\tw$ that of $\rc X$.
\ep
\bpr
We find in \cite[5.7.4]{MR2273730} the  homeomorphism
$
h \colon \rc (  \S^{m}*X) \to \rc \S^m \times \rc X:
$
defined by
\begin{equation}\label{refirefi}
h([[[z,t],y],r]) =
\left\{
\begin{array}{ll}
  ( [z,2rt],[y, r])& \hbox{si } t\leq 1/2\\[.3cm]
  ([z,r], [y,2r(1-t)])  & \hbox{si } t\geq 1/2 .
\end{array}
\right.
\end{equation}
Let us verify that $h$ preserves the stratifications.
We write $ \tu$ the apex of the cone $\rc (S^m *X)$.
We distinguish three cases.

\smallskip

+ $h(\tu) = (\tv,\tw). $

+ $ h (\S^{m} \times ]0,1[ ) = (\rc \S^m\times \{ \tw \}) \menos \{ (\tv,\tw)\} $ since $h([[[z,1],y],r]) = ([z,r], \tv)$.

+ The restriction 
$ h\colon \rc \S^m  \times X \times ]0,1[  \to \rc \S^m  \times X \times ]0,1[ $ is given by  
$$h([z,t],y,r)\mapsto 
\left\{
\begin{array}{ll}
  ( [z,2rt],y, r)& \hbox{si } t\leq 1/2\\[.3cm]
 ([z,r], y,2r(1-t))   & \hbox{si } t\geq 1/2 .
\end{array}
\right.
$$
It is  clearly a   stratified homeomorphism.
\epr\negro

\subsection{Perversities.} \label{subsec:perv} Consider $(X,\SS)$ a filtered space.
A \emph{$M$-perversity}, or simply \emph{perversity}, on $(X,\SS)$ is a map $\ov p \colon \SS \to  \ov \Z = \Z \sqcup \{-\infty,\infty\}$ verifying $\ov p(S)=0$ for any regular stratum \cite{RobertSF}. 

The \emph{top perversity} is the perversity defined by $\ov t (S) =\codim S - 2$ for each singular stratum $S$. The \emph{dual perversity} of  $\ov p$ is the perversity $D\ov p$ defined by $D\ov p = \ov t - \ov p$.

A perverse space is a triple $(X,\SS,\ov p)$ where $(X,\SS)$ is a filtered space and $\ov p$ is a perversity on $(X,\SS)$. 
Given a stratified map $f \colon (X,\SS) \to (Y,\TT)$, a perversity $\ov q$ on $(Y,\TT)$ and a perversity  $\ov p$ on $(X,\SS)$,  we define 

$\bullet$ the \emph{pull-back perversity} $f^\star\ov q$ on $(X,\SS)$ by: $f^\star\ov q (S) = \ov q (\bi Sf)$ for each $S \in \SS^{sing}$,
 
 $\bullet$ the \emph{push-forward perversity} $f_\star\ov p$ on $(X,\T)$ by: $f_\star\ov p (T) = \min\{ \ov p ( Q ) \mid \bi Qf = T\}$ for each $T \in \TT^{sing}$, with $\inf \emptyset =\infty$.

Notice that  $f^\star f_\star \ov p \leq  \ov p$ when $\ov p(S)\geq 0$ if $\bi SI \in \TTreg$.

\bigskip

We make a quick reminder of the intersection homologies/cohomologies deployed in this work. They have been mainly introduced in order to study the Poincar\'e duality of the intersection (co)homology in different contexts.

\subsection{Tame and intersection (co)homologies}
\label{tameinter} (cf. \cite{CST7,CST3}).  We fix  an $n$-dimensional  perverse space $(X,\SS,\ov p)$.
Tame intersection homology is a variant of the classic intersection homology (cf. \cite{MR572580,MR696691,MR800845}).
When the perversity $\ov p$ is greater than the top perversity it is possible to have a $\ov p$-intersection chains contained in the singular part $\Sigma$ of $X$. This fact prevents the Poincar\'e duality and the de Rham Theorem. For this reason the tame intersection homology was introduced (cf. \cite{CST3,CST7,LibroGreg,MR2210257}).

A \emph{filtered simplex} is a singular simplex $\sigma\colon\Delta\to X$ where the euclidean simplex $\Delta$ is endowed with a filtration  
 $\Delta=\Delta_{0}\ast\Delta_{1}\ast\cdots\ast\Delta_{n}$, called \emph{$\sigma$-decomposition},
verifying
$
\sigma^{-1}X_{i} = \Delta_{0}\ast\Delta_{1}\ast\cdots\ast\Delta_{i},
$
for each~$i \in \{0, \ldots, n\}$. 
 A factor $\Delta_{i}$ can be empty with the convention  $\emptyset * Y=Y$. 
 The filtered simplex $\sigma$ is  a \emph{regular simplex} when $\im \sigma \not \subset\Sigma$, that is, $\Delta_n \ne \emptyset$.
 
We decompose the boundary of a filtered simplex $\Delta=\Delta_{0}\ast\dots\ast\Delta_{n}$  as
$\partial\Delta=\partial_{\reg}\Delta+\partial_{\sing}\Delta,$
where $\partial_{\reg}\Delta$ contains all the regular simplices.
% .

The \emph{perverse degree} of the filtered simplex $\sigma$ relatively to a stratum $S \in \dos {\SS} {\mathcal F}$
is 
$$
\|\sigma\|_{S}=\left\{
 \begin{array}{cl}
 -\infty,& \text{if } S \cap \im \sigma=\emptyset,\\
 \dim (\Delta_{0}\ast\cdots\ast\Delta_{\dim S}),& \text{otherwise.}
  \end{array}\right.
  $$
%Consider a filtered chain $\xi=\sum_{j\in J}n_{j}\sigma_{j}$, $n_{j}\in G$, that is, made up of filtered simplices.
%The \emph{perverse degree} of $\xi$ 
% along the stratum $S$ is given by
%  $\|\xi\|_{S}=\max_{j\in J}\|\sigma_{j}\|_{S}.$
%

A filtered simplex $\sigma \colon \Delta \to X$ is \emph{$ \ov{p}$-allowable} if
$
\|\sigma\|_{S}\leq \dim \Delta - \codim S + \ov{p}(S)$, for each  $S\in {\SS}$.
Moreover, if $\im \sigma \not\subset \Sigma$ then the simplex $\sigma$ is said to be \emph{$\ov p$-tame}.

The chain complex $\lau C{\ov{p}}*{X;\SS}$ is the $G$-module formed of the linear combinations
$\xi=\sum_{j \in J}n_{j}\sigma_{j}$, where each  $\sigma_{j}$ is $ \ov{p}$-allowable, 
and such that
${\partial}\xi=\sum_{\ell \in L}n_{\ell}\tau_{\ell}$, where the simplices
$\tau_\ell$ are $\ov{p}$-allowable.
We call $(\lau C{\ov{p}}*{X;\SS},\partial)$ the \emph{$\ov{p}$-intersection complex} and its homology,
$\lau H{\ov{p}}*{X;\SS}$, the
\emph{$\ov{p}$-intersection homology.} %
This designation is justified since this homology  matches with the usual intersection homology (cf. \cite[Theorem A]{CST3}).

The chain complex $\lau \gC {\ov{p}}* {X;\SS}$ is the $G$-module formed of the linear combinations
$\xi=\sum_{j \in J}n_{j}\sigma_{j}$, with each $\sigma_{j}$ is $\ov{p}$-tame, 
and such that
${\partial_{\reg}}\xi=\sum_{\ell \in L}n_{\ell}\tau_{\ell}$, where the simplices
$\tau_{\ell}$ are $\ov{p}$-tame.
We call $(\lau \gC {\ov{p}}* {X;\SS},\gd={\partial_{\reg}})$ the \emph{tame $\ov{p}$-intersection complex} and its homology,
$\lau \gH {\ov{p}}* {X;\SS}$, the
\emph{tame $\ov{p}$-intersection homology.} %
This designation is justified since this homology  matches with the usual tame intersection homology (cf. \cite[Theorem B]{CST3}).
We have $\lau H {\ov p} * {X;\SS} = \lau \gH {\ov p} * {X;\SS}$ when $\ov p \leq \ov t$ (cf. \cite[Remark 3.9]{CST3}).

Associated cohomology is defined by using the functor $\Hom$, as usual in algebraic topology. We put the dual complexes 
$\lau C *{\ov{p}} {X;\SS}= \hom (\lau C {\ov{p}}*{X;\SS};\SS)$
and
$\lau \gC *{\ov{p}} {X;\SS}= \hom (\lau \gC{\ov{p}}*{X;\SS};\SS)$
endowed with the dual differential $d$. Their  cohomologies are the  \emph{$\ov p$-intersection cohomology} 
$\lau H *{\ov{p}} {X;\SS}$ and the   \emph{$\ov p$-tame intersection cohomology} 
$\lau \gH *{\ov{p}} {X;\SS}$.

 Let  $U \subset X$ be an open subset .
The relative homologies  $\lau H  {\overline{p}}  *{X,U;\SS}$ and $\lau \gH  {\overline{p}}  *{X,U;\SS}$ are defined from quotient complexes
$
\lau C  {\overline{p}}  *{X,U;\SS}
= 
\lau C  {\overline{p}}  *{X;\SS}/\lau C  {\overline{p}} * {U;\SS}
$
and
$
\lau \gC  {\overline{p}}  *{X,U;\SS}
= 
\lau \gC  {\overline{p}}  *{X;\SS} / \lau \gC  {\overline{p}} * {U;\SS}
$
(\cite[Definition 4.5]{CST3}).
The relative cohomologies $\lau H *  {\overline{p}}  {X,U;\SS}$ and $\lau \gH  *{\overline{p}}  {X,U;\SS}$ are defined by using the functor $\Hom$ (cf. \cite[Definition 7.1.1]{LibroGreg}).

The \emph{(tame) intersection cohomology with compact supports} are defined by
\begin{eqnarray}
\label{iccs}
\lau H * {\ov p,c} {X;\SS} = \varinjlim_{K \subset X} \lau H * {\ov p} {X,X\menos K;\SS}
\phantom{--} \hbox{ and }
\phantom{--}
\hfill \lau \gH * {\ov p,c} {X;\SS} = \varinjlim_{K \subset X} \lau \gH * {\ov p} {X,X\menos K;\SS},
\end{eqnarray}
where $K$ runs over the compact subsets of $X$ (cf. \cite[Definition 6.1]{MR3046315}).

\subsection{Main properties for (tame) intersection (co)homology} \label{mpti} We group here the main properties  of  the (tame) intersection homology.
We fix a perverse set $(X,\SS,\ov p)$. 

\bigskip

 {a. \bf Mayer-Vietoris.} Associated to  an  open cover $\{U,V\}$ of $X$ we have the long exact sequences
$$
\cdots \to  \lau H {\ov{p}}{k+1}  {X;\SS} \to \lau H {\ov{p}}k  {U\cap V;\SS}\to
\lau H {\ov{p}}k  {U;\SS}\oplus \lau H {\ov{p}}k  {V;\SS}\to
 \lau H {\ov{p}}k  {X;\SS}\to \cdots  ,
$$
$$
\cdots \to  \lau \gH {\ov{p}}{k+1}  {X;\SS} \to \lau \gH {\ov{p}}k  {U\cap V;\SS}\to
\lau \gH {\ov{p}}k  {U;\SS}\oplus \lau \gH {\ov{p}}k  {V;\SS}\to
 \lau \gH {\ov{p}}k  {X;\SS}\to \cdots  ,
$$
(cf. \cite[Proposition 4.1]{CST3}).

\medskip
%
%\hspace{0cm} {\bf 2. Stratified homotopy.} We have the isomorphisms 
%$
%\lau H {\ov{p}}*  {\R \times X,\I \times \SS} = \lau H {\ov{p}}*  { X;\SS},
%$ and $
% \lau \gH {\ov{p}}*  {\R \times X,\I \times \SS} = \lau \gH {\ov{p}}*  { X, \SS}
% $
%(cf. \cite[Corollaire 3.9, Corollaire 7.8]{CST3}).
%
%\medskip

 {b. \bf Local calculations.} We have the isomorphisms 
$
\lau H {\ov p} * {\R^m \times X, \I \times \SS} =\lau H {\ov p} * { X,  \SS} $
and
$
\lau  \gH  {\ov{p}} * {\R^m \times  X,\I \times \SS }=
\lau  \gH  {\ov{p}} * { X, \SS }
$
 where the  isomorphism  comes from the canonical projection $\pr \colon  \R^m \times X \to X$.
(cf. \cite[Corollary 3.14]{CST3}). 
If $L$ is compact, we have the isomorphisms
\begin{eqnarray*}
\lau H {\ov p} k {\R^m \times \rc L, \I \times \rc \SS} &=&
\left\{
\begin{array}{cl}
\lau H {\ov p} k {L,  \SS} &\text{ if } k\leq D\ov{p}( \tv),\\
0&\text{ if } 0\neq k> D\ov{p}( \tv),\\
G&\text{ if } 0=k>D\ov{p}( \tv),
\end{array}\right.
\\[,2cm]
\lau  \gH  {\ov{p}} k {\R^m \times  \rc L,\I \times \rc\SS }&=&
\left\{
%:
\begin{array}{cl}
\lau \gH  {\ov{p}} k {L;\SS}&\text{ if } k\leq D\ov{p}( \tv)
,\\
0&\text{ if }  k >  D\ov{p}( \tv)
\end{array}\right.
\end{eqnarray*}
where  the  isomorphisms (first line) come from the inclusion $\iota \colon L \to \R^m \times \rc L$ defined by $\iota (x) = (0, [x,1/2])$ (cf. \cite[Proposition 5.1]{CST3}).

\medskip

\rojo 
c. {\bf  Join. Suspension.} 

\bl \label{CalJoin}
Let $(X,\SS)$ be a  CS-set. Consider $\ov p
$ a perversity on the join $\S^m*X$  for some  integer $m\in \N$ (cf.  \exemref{Ejemplos}). We suppose that $\ov p$ takes the same value $\ov p(\S^m)$ on each connected component $(\S^m)_{cc}$ of $\S^m$. We have
%\begin{eqnarray*}
%\lau H {\ov p} k {\S^m * X,  \SS_{\star m}} &=&
%\lau H {\ov p} k { \rc X, \rc \SS} \oplus \lau H {\ov p} {k-m} { \rc X,X, \rc\SS} 
%\\
%\lau  \gH  {\ov{p}} k {  \S^m * X,\SS_{\star m}}&=&
%\lau \gH  {\ov{p}} k {\rc X, \rc \SS}
%\oplus \lau \gH {\ov p} {k-m} { \rc X,X,  \rc \SS} 
%\\
%\lau \IH k {\ov p}  {S^m * X,  \SS_{\star m+1}}& =&
%\lau \IH  k {\ov p}  { \rc X, \rc\SS} 
%\oplus
%\lau \IH  {k-m} {\ov p} { \rc X, X, \rc \SS} 
%\end{eqnarray*}
\begin{eqnarray*}
\lau H {\ov p} k {\S^m * X,  \SS_{\star m}} &=&
\left\{
\begin{array}{cl}
\lau H {\ov p} k { X, \SS} &\text{ if } k\leq D\ov{p}(\S^m),
\\
G &\text{ if } 0 = k> D\ov{p}(S^m),\\
0&\text{ if } D\ov{p}(\S^m) < k \leq D\ov{p}(\S^m)  +m +1 , k\ne 0
\\
\lau {\widetilde H} {\ov p} {k-m-1} { X, \SS} &\text{ if } k\geq D\ov{p}(\S^m)  +m+2, k\ne  0
\end{array}\right.
\\[.2cm]
\lau  \gH  {\ov{p}} k {  \S^m * X,\SS_{\star m}}&=&
\left\{
%:
\begin{array}{cl}
\lau \gH  {\ov{p}} k {X;\SS}&\text{ if } k\leq D\ov{p}(\S^m),
,\\
0&\text{ if } D\ov{p}(\S^m) < k \leq D\ov{p}(\S^m)  +m +1,
\\
\lau \gH {\ov p} {k-m-1} { X, \SS} &\text{ if } k\geq D\ov{p}(\S^m)  +m+2,
\end{array}\right.
\end{eqnarray*}
where  the first line isomorphisms come from the natural inclusion $X \hookrightarrow  \S^m * X$ (cf. \exemref{Ejemplos}).
\el
\bpr
The three cases use the same technics: Mayer-Vietoris and local calculations. We prove just the first one.
The join stratification is 
$
\SS_{\star m } =\{(\S^{m})_{cc}, \rc \S^{m} \times  S \mid S \in \SS\}$.
The perversity $\ov p$ is determinate by the number $\ov p (\S^{m})$ and a perversity on $X$, still denoted by $\ov p$, related by: $\ov p(\rc \S^{m} \times S) = \ov p(S)$ for each $S \in \SS$. 

We consider the open covering $\S^m * X = U \cup V$ where $U =( \S^m * X ) \menos \S^m = \rc \S^m \times X$ and $V = (\S^m * X ) \menos X = \S^m \times \rc X$. This last equality comes from the homeomorphism $[[z,t],y] \mapsto (z,[y,1-t])$.
 We have $U \cap V = \S^m \times ]0,1[ \times X$.
The inclusion $U\cap V \hookrightarrow U$ induces the morphism 
$\pr_* \colon \lau H {\ov p} * {\S^m \times  X,  \I \times \SS} \to
\lau H {\ov p} k { X,  \SS}$, where $\pr \colon\S^m \to X$ is la projection canonique,
and
the inclusion $U\cap V \hookrightarrow V$ induces the morphism 
$I_* \colon \lau H {\ov p} * {\S^m \times  X,  \I \times \SS} \to
\lau H {\ov p} * { \S^m \times \rc X,  \I \times \rc \SS}$, where $I \colon \S^m \times X \hookrightarrow \S^m \times \rc X$ 
is the inclusion $(\theta,x) \mapsto (\theta,[x,1/2])$ (cf. \cite[Corollary 3.14]{CST3}).

Applying Mayer-Vietoris (cf. \cite[Proposition 4.1]{CST3}) we get the short exact sequence
$$
0 \to  \coker F_{k}  \to \lau H {\ov p} k {\S^m * X,  \SS_{\star m}} \to  \ker F_{k-1} \to 0,
$$
for each $k\in \Z$, where
$$
F_k  = \pr_k \oplus I_k \colon  \lau H {\ov p} k {\S^m \times  X,  \I \times \SS}  \to  \lau H {\ov p} k {X, \SS}
\oplus
\lau H {\ov p} k { \S^m \times \rc X,  \I \times \rc \SS}.
$$
Since intersection homology verifies Kunneth (cf. \cite[Proposition 4.7]{CST3}) then this maps becomes
$$
F_k =    \lau H {\ov p} k {X,   \SS}  \oplus \lau H {\ov p} { k-m} {X, \SS} \to  \lau H {\ov p} k {X, \SS}
\oplus
\lau H {\ov p} k { \rc X,   \rc \SS}\oplus
\lau H {\ov p} {k-m} { \rc X,  \rc \SS}
$$
where 
$F_k(a,b) =\left\{
\begin{array}{ll}
(a,J_k(a), J_{k-m}(b))& \hbox{if } m\ne 0 \\
(a-b,J_k(a), J_{k}(b))& \hbox{if } m=0
\end{array}
\right.
$.
 Here, $J_\ell \colon   \lau H {\ov p} \ell {X,   \SS} \to  \lau H {\ov p} \ell {\rc X,   \SS} $ is  induced by the inclusion $x \mapsto [x,1/2]$.
 Notice that
 $\ker F_k = \ker J_{k-m}
 $
 and $\coker F_k = \lau H {\ov p} {k} { \rc X,  \rc \SS} \oplus \coker J_{k-m}$.
 
 Let us consider the long exact sequence associated to the pair $(\rc X,X)$ 
$$
 \cdots      \lau H {\ov p} {k-1} {\rc X,  X, \rc \SS}  \to  \lau H {\ov p} k { X,    \SS}  \xrightarrow  {J_*}  \lau H {\ov p} k {\rc X,   \rc \SS}  \rightarrow
    \lau H {\ov p} k {\rc X,  X, \rc \SS}  \to\cdots
    $$
(cf. \cite[Definition 4.4]{CST3}). 
Notice that for each  $k\in \Z$ 
\begin{itemize}
\item[-] the map $F_{k-1}$ is a monomorphism  or the map $F_k$ is an epimorphism and
\item[-] the map $J_{k-1}$ is a monomorphism  or the map $J_k$ is an epimorphism 
\end{itemize}
We conclude that
\begin{eqnarray*}
\lau H {\ov p} k {\S^m * X,  \SS_{\star m}} &=&
\coker F_k \oplus \ker F_{k-1} =
\lau H {\ov p} k { \rc X, \rc \SS} \oplus 
\coker J_{k-m} \oplus \ker J_{k-m-1}  \\
&=& \lau H {\ov p} k { \rc X, \rc \SS} \oplus\lau H {\ov p} {k-m} { \rc X,X, \rc\SS} 
\\
&=_{(1)}& \left\{
\begin{array}{ll}
\lau H {\ov{p}} {k}{X,\SS}& \hbox{if } k\leq D\ov p (\S^m)\\
0 & \hbox{if } 0\ne k  > D\ov p (\S^m)  \\
G& \hbox{if } 0= k > D\ov p (\S^m) 
\end{array}
\right.
\oplus \left\{
%:
\begin{array}{ll}
0&\text{ if }  k \leq  D \ov p(\S^m)+ m +1\\
\lau {\widetilde{H} }{\ov{p}} {k-m-1}{ X, \SS}
&\text{ if } k \geq  D \ov p(\S^m)+m +2
\end{array}\right.
\end{eqnarray*}
which gives the claim. Here, ${(1)}$ is given by \cite[Proposition 5.2, Corollary 5.3]{CST3}.

When $k\leq D \ov p(\S^m)$ the isomorphism $\lau H  {\ov p} k {X,\SS} = \lau H {\ov p} k {\S^m * X,  \SS_{\star m}} $ 
comes from the map $X \xrightarrow f U \xrightarrow g \S^m * X$ with $f(x) = ([z,0],x)$ and $g([z,t],x) = [[z,t],x]$. In other words, the natural inclusion $X \hookrightarrow \S^m*X$, $x \mapsto  [[z,0],x]$ (cf. \exemref{Ejemplos}). 
\epr 
%Using the above calculation (see also \cite[Lemma 3.6]{CST8}), one gets:
%\begin{eqnarray*}
%\lau H {\ov p} k {\S^m * X,  \SS_{\star m}} &=&
%\left\{
%\begin{array}{cl}
%\lau H {\ov p} k { X, \SS} &\text{ if } k\leq D\ov{p}(\S^m),\\
%G &\text{ if } 0 = k> D\ov{p}(\S^m),
%\\
%0&\text{ if } D\ov{p}(\S^m) < k \leq D\ov{p}(\S^m)  +m +1 , k\ne 0
%\\
%\lau {\widetilde H} {\ov p} {k-m-1} { X, \SS} &\text{ if } k\geq D\ov{p}(\S^m)  +m+2, k \ne 0
%\end{array}\right.
%\\[.2cm]
%\lau  \gH  {\ov{p}} k {  \S^m * X,\SS_{\star m}}&=&
%\left\{
%%:
%\begin{array}{cl}
%\lau \gH  {\ov{p}} k {X;\SS}&\text{ if } k\leq D\ov{p}(\S^m),
%%
%,\\
%0&\text{ if } D\ov{p}(\S^m) < k \leq D\ov{p}(\S^m)  +m +1,
%\\
%\lau \gH {\ov p} {k-m-1} { X, \SS} &\text{ if } k\geq D\ov{p}(\S^m)  +m+2,
%\end{array}\right.
%\end{eqnarray*}
%where  the isomorphism comes from the inclusion $\iota \colon X \to \S^m * X$ defined by $\iota (x) = [[\bullet,0], x]$
Let us look at the case $m=0$, that is, the suspension $\Sigma X$ when $\S^0 = S^0$. Previous calculations suppose that the perversity $\ov p$ takes the same value at the two connected components of $\S^0$: $\tn$ and $\ts$. In the general case, if $\ov p(\ts) \geq \ov p(\tn)$, we have
\begin{eqnarray*}
\lau H {\ov p} k {\S^0 * X,  \SS_{\star 0}} &=&
\left\{
\begin{array}{cl}
\lau H {\ov p} k { X, \SS} &\text{ if } k\leq D\ov{p}(\ts),\\
G &\text{ if } 0 = k> D\ov{p}(\ts),
\\
0&\text{ if } D\ov{p}(\ts) < k \leq D\ov{p}(\tn)  +1 , k\ne 0
\\
\lau {\widetilde H} {\ov p} {k-1} { X, \SS} &\text{ if } k\geq D\ov{p}(\tn) + 2, k \ne 0
\end{array}\right.
\\[.2cm]
\lau  \gH  {\ov{p}} k {  \S^0 * X,\SS_{\star 0}}&=&
\left\{
%:
\begin{array}{cl}
\lau \gH  {\ov{p}} k {X;\SS}&\text{ if } k\leq D\ov{p}(\ts),
,\\
0&\text{ if } D\ov{p}(\ts) < k \leq D\ov{p}(\tn)  +1,
\\
\lau \gH {\ov p} {k-1} { X, \SS} &\text{ if } k\geq D\ov{p}(\tn)  +2,
\end{array}\right.
\end{eqnarray*}
\negro
\medskip

d. {\bf Relative homologies}. Let $U$ be an open subset of $X$.
We have the associated long exact sequences for homology
$$
\cdots \to \lau H  {\overline{p}}  k{U;\SS}
\to \lau H  {\overline{p}}  k{X;\SS}
\to \lau H  {\overline{p}}  k{X,U;\SS}
\to \lau H  {\overline{p}}  {k-1}{U;\SS}
\to \cdots ,
$$
$$
\cdots \to \lau \gH  {\overline{p}}  k{U;\SS}
\to \lau \gH  {\overline{p}}  k{X;\SS}
\to \lau \gH  {\overline{p}}  k{X,U;\SS}
\to \lau \gH  {\overline{p}}  {k-1}{U;\SS}
\to \cdots 
$$
(cf. \cite[Definition 4.5]{CST3})
and for cohomology
$$
\cdots \to \lau H  k {\overline{p}}  {X,U;\SS}
\to \lau H   k{\overline{p}}  {X;\SS}
\to \lau H k  {\overline{p}}  {U;\SS}
\to \lau H  {k+1}{\overline{p}}  {X,U;\SS}
\to \cdots 
$$
$$
\cdots \to \lau \gH  k {\overline{p}}  {X,U;\SS}
\to \lau \gH   k{\overline{p}}  {X;\SS}
\to \lau \gH k  {\overline{p}}  {U;\SS}
\to \lau \gH  {k+1}{\overline{p}}  {X,U;\SS}
\to \cdots 
$$
(cf. \cite[Theorem 7.1.11]{LibroGreg})\footnote{Only the tame case is considered in this reference but the non-tame case can be treated in the same way.}.

\medskip

e. {\bf  Universal Coefficients Theorem.} 
There are two natural exact sequences
$$
0 \to \Ext(\lau H  {\ov p}  {k-1}{X;\SS}, R) \to  \lau H k {\ov p} {X;\SS} \to \Hom(\lau H {\ov p}  {k}{X;\SS}, R) \to 0,
$$
$$
0 \to \Ext(\lau \gH  {\ov p}  {k-1}{X;\SS}, R) \to  \lau \gH k {\ov p} {X;\SS} \to \Hom(\lau \gH {\ov p}  {k}{X;\SS}, R) \to 0,
$$
for every $k\in \N$. We find the proof of the second assertion in \cite[Proposition 7.1.4]{LibroGreg}. But the proof is the same for the first sequence.

\subsection{Intersection  homology from Borel-Moore point of view  (cf. \cite{MR2276609,ST1})} 
The  \emph{Borel-Moore $\ov p$-intersection homology} $\lau H {BM,\ov p} * {X;\SS}$
and the \emph{Borel-Moore $\ov p$-tame intersection homology} $\lau \gH {BM,\ov p} * {X;\SS}$
are defined in the same way as  the homologies defined in \ref{tameinter} have been defined but considering locally finite chains instead  of finite chains.

When $X$ is compact, we have    $\lau H {BM,\ov p} * {X;\SS} =  \lau H {\ov p} * {X;\SS}$ and $\lau \gH {BM,\ov p} * {X;\SS} =  \lau \gH {\ov p} * {X;\SS}$.

\subsection{Main properties for Borel-Moore (tame) intersection homology} \label{mpbmti} 
 We suppose that $X$ is a {\em hemicompact }space, that is, there exists an increasing sequence of compact subsets $K_0 \subset K_1 \subset \cdots K_n \subset \cdots$ such that,  each compact $K \subset X$ is included on some $K_n$. We have proved in \cite[Proposition 2.2]{CST8} that the Borel-Moore intersection homology can be computed in terms of the intersection homology in the following way\footnote{In the op.cit. the result is proved for the Borel-Moore tame intersection homology. The same proof works for the Borel-Moore intersection homology.}:
\be \label{limpro}
\displaystyle \lau H {BM,\ov p}* {X} =  \varprojlim_{n \in \N} \lau H {\ov p}* {X,X\menos K_n} \ \ \ \hbox{ and } \ \ \ 
\displaystyle \lau \gH {BM,\ov p}* {X} =  \varprojlim_{n \in \N} \lau \gH {\ov p}* {X,X\menos K_n}.
\ee

\subsection{Blown-up intersection cohomologies (cf. \cite{CST5})} 
Let $\hiru N{*}\Delta$ 
and  $\Hiru N*\Delta$ be the  simplicial chain and cochain complexes
of  an euclidean simplex $\Delta$, with coefficients in $R$. 
For each simplex $F \in \hiru N{*}\Delta$, we write $\1_{F}$ the element of $\Hiru N{*}\Delta$ taking the value 1 on $F$ and 0 otherwise. 
Given a face $F$  of $\Delta$, we denote by $(F,0)$ the same face viewed as face of the cone $\tc\Delta=\Delta\ast [\tw]$ and by $(F,1)$ 
the face $\tc F$ of $\tc \Delta$. 
Here, $[\tw]=(\emptyset,1)=\tc \emptyset$  is the apex o f the cone $\tc \Delta$.
Cochains on the cone are denoted $\1_{(F,\varepsilon)}$ for $\varepsilon =0$ or $1$.
For defining the blown-up intersection complex, we first set
$$\Hiru\tN{*}\Delta=N^*(\tc\Delta_0)\otimes\dots\otimes N^*(\tc\Delta_{n-1})\otimes N^*(\Delta_n).$$
A basis of $\Hiru \tN{*}\Delta$ is composed of the elements 
$\1_{(F,\varepsilon)}=\1_{(F_{0},\varepsilon_{0})}\otimes\dots\otimes \1_{(F_{n-1},\varepsilon_{n-1})}\otimes \1_{F_{n}}$,
 where 
$\varepsilon_{i}\in\{0,1\}$ and
$F_{i}$ is a face of $\Delta_{i}$ for $i\in\{0,\dots,n\}$ or the empty set with $\varepsilon_{i}=1$ if $i<n$.
We set
$|\1_{(F,\varepsilon)}|_{>s}=\sum_{i>s}(\dim F_{i}+\varepsilon_{i})$,
with the convention $\dim \emptyset=-1$.

%%%%%%%%%%%%%%%%%%%%%%%
%%%%%%%%%%%%%%%%

Let $\ell\in \{1,\ldots,n\}$ and
$\1_{(F,\varepsilon)}
\in\Hiru\tN{*}\Delta$.
The  \emph{$\ell$-perverse degree} of 
$\1_{(F,\varepsilon)}\in \Hiru N{*}\Delta$ is
$$
\|\1_{(F,\varepsilon)}\|_{\ell}=\left\{
\begin{array}{ccl}
-\infty&\text{if}
&
\varepsilon_{n-\ell}=1,\\
|\1_{(F,\varepsilon)}|_{> n-\ell}
&\text{if}&
\varepsilon_{n-\ell}=0.
\end{array}\right.$$
Given $\omega = \sum_b\lambda_b \ \1_{(F_b,\varepsilon_b) }\in\Hiru \tN{*}\Delta$ with 
$0 \ne \lambda_{b}\in R$ for all $b$,
the \emph{$\ell$-perverse degree} is
$$\|\omega \|_{\ell}=\max_{b}\|\1_{(F_b,\varepsilon_b)}\|_{\ell}.$$
By convention, we set $\|0\|_{\ell}=-\infty$.

Let $\sigma\colon \Delta=\Delta_0\ast\dots\ast\Delta_n\to X$ be a filtered simplex.
We set $\tres\tN *{\sigma}=\Hiru \tN{*}\Delta$.
If $\delta_{\ell}\colon \Delta' 
\to\Delta$ 
 is  an inclusion of a face of codimension~1,
  we denote by
$\partial_{\ell}\sigma$ the filtered simplex defined by
$\partial_{\ell}\sigma=\sigma\circ\delta_{\ell}\colon 
\Delta'\to X$.
If $\Delta=\Delta_{0}\ast\dots\ast\Delta_{n}$ is filtered, the induced filtration on $\Delta'$ is denoted
$\Delta'=\Delta'_{0}\ast\dots\ast\Delta'_{n}$.
The \emph{blown-up intersection complex} of $X$ is the cochain complex 
$\Hiru \tN *X$ composed of the elements $\omega$
associating to each regular filtered simplex
 $\sigma\colon \Delta_{0}\ast\dots\ast\Delta_{n}\to X$
an element
 $\omega_{\sigma}\in \tres \tN*{\sigma}$  
such that $\delta_{\ell}^*(\omega_{\sigma})=\omega_{\partial_{\ell}\sigma}$,
for any face operator
 $\delta_{\ell}\colon\Delta'\to\Delta$
 with $\Delta'_{n}\neq\emptyset$. 
 The differential $d \omega$ is defined by
 $(d \omega)_{\sigma}=d(\omega_{\sigma})$.
 The \emph{perverse degree of $\omega$ along a singular stratum $S$} equals
 $$\|\omega\|_S=\sup\left\{
 \|\omega_{\sigma}\|_{\codim S}\mid \sigma\colon \Delta\to X \;
 \text{regular simplex such that }
\im \sigma \cap S\neq\emptyset
 \right\}.$$
 We denote $\|\omega\|$ the map $S \mapsto || \omega ||_S$, where $||\omega||_S=0$ if $S$ is a regular stratum.
 A \emph{cochain $\omega\in \Hiru \tN * {X;\SS}$ is $\ov{p}$-allowable} if $\| \omega\|\leq \ov{p}$ 
 and of \emph{$\ov{p}$-intersection} if $\omega$ and $d\omega$ are $\ov{p}$-allowable. 
 We denote $\lau \tN*{\ov{p}}{X;\SS}$ the complex of $\ov{p}$-intersection cochains and  
 $\lau \crH {\ov{p}}*{X;\SS}$ its homology called
 \emph{blown-up intersection cohomology}  of $X$ 
 for the perversity~$\ov{p}$. 
 
 A    subset $K\subset X$ is a \emph{support } of
the cochain $\omega \in \lau \tN * {\ov p} {X;\SS}$ if $\omega_\sigma = 0$, for any regular simplex $\sigma \colon \Delta \to X$ such that $ \im \sigma \cap K = \emptyset$. We also say that $\omega \equiv 0$ on $X\menos K$.

 We
denote $\lau \tN * {\ov p,c} {X;\SS}$ the complex of $\ov p$-intersection cochains with compact supports and 
$\lau \IH * {\ov p,c} {X;\SS}$ its cohomology.
%%%%%%%%%%%%%%%%%%%%%%%%%%%%%%%%%%%eo
\subsection{Main properties for blown up intersection cohomologies} \label{mpbi} We group here the main properties  of blown-up intersection cohomology.
We fix a perverse space $(X,\SS,\ov p)$.

\medskip

a. {\bf  Mayer-Vietoris.} Suppose $X$ paracompact. Given an open cover $\{U,V\}$ of $X$ we have the long exact sequence (cf. \cite[Corollary 10.1]{CST5})
{\small $$
\cdots \to 
\lau \IH  k{\ov{p}}  {X;\SS}\to 
\lau \IH k  {\ov{p}}  {U;\SS}\oplus \lau \IH  *  {\ov{p}}  {V;\SS}\to
\lau \IH k {\ov{p}}  {U\cap V;\SS}\to
 \lau \IH  {k+1}{\ov{p}}  {X;\SS}\to  \cdots.
$$

\medskip

b.  {\bf  Local calculations.} We have the isomorphism
$$
\lau \IH k {\ov p}  {\R^m \times  X, \I \times \rc \SS} = \lau \IH k {\ov p}  {  X, \I },
$$
coming from the inclusion  $\iota \colon   X \to \R^m \times X$ defined by $\iota(x) = (0,x)$
(cf.  \cite[Theorem D]{CST5}). If $L$ is compact, we have the isomorphism
$$
\lau \IH k {\ov p}  {\R^m \times \rc L, \I \times \SS} =
\left\{
\begin{array}{cl}
\lau \IH k {\ov p}  {L,  \SS} &\text{ if } k\leq \ov{p}( \tv),\\
0&\text{ if }  k> \ov{p}( \tv),\\
\end{array}\right.
$$
%$$
%\lau \IH k {\ov p,c}  {\R^m\times \rc X, \I \times \rc \SS} =
%\left\{
%\begin{array}{cl}
%\lau \IH {k-m-1} {\ov p}  { X,  \SS} &\text{ if } k\geq \ov{p}( \tv) +2,\\
%0&\text{ if } k\leq  \ov{p}( \tv) +1,\\
%\end{array}\right.
%$$
where  the isomorphism (first line) comes from the inclusion $\iota \colon L \to \R^m \times \rc L$ defined by $\iota (x) = (0, [x,1/2])$
(cf.  \cite[Theorem  E]{CST5}).
%,\cite[Proposition 2.14, Proposition 2.15]{CST7}).

\medskip

\rojo
c. {\bf Join.} Proceeding as in \lemref{CalJoin},  one gets the isomorphism:$$
\lau \IH k {\ov p}  {\S^m * X,  \SS_{\star m+1}} =
\left\{
\begin{array}{cl}
\lau \IH  k {\ov p}  { X, \SS} &\text{ if } k\leq \ov{p}(\S^m),\\
0&\text{ if } \ov{p}(\S^m) < k \leq \ov{p}(\S^m)  +m+1 ,\\
\lau \IH  {k-m-1} {\ov p} { X, \SS} &\text{ if } k\geq \ov{p}(\S^m)  +m+2,
\end{array}\right.
$$
where  the first line isomorphism comes from the natural inclusion $X \hookrightarrow  \S^m * X$ (cf. \exemref{Ejemplos}).
\negro

\medskip

d. {\bf Relative cohomology.} We consider an open subset $U \subset X$.
The \emph{complex of relative $\ov{p}$-intersection cochains} is 
$\lau \tN * {\ov p}{X,U;\SS} =  \lau \tN * {\ov p}{X;\SS}  \oplus \lau \tN {*-1} {\ov p}{U;\SS} $, endowed with the differential 
 $
D(\alpha,\beta) = (d\alpha, \alpha - d\beta).
 $
Its homology is the
\emph{relative blown-up $\ov{p}$-intersection cohomology of the perverse pair $(X,U,\ov{p})$,}
denoted by 
$\lau \IH{*}{\ov{p}}{X,U}$.

By definition, we have a long exact sequence associated to the perverse pair $(X,U,\ov{p})$:
$$
\ldots\to
\lau \IH {i}{\ov{p}}{X;\SS}\stackrel{\i^\star  }{\to}
\lau \IH {i}{\ov{p}}{U;\SS}\to
\lau \IH{i+1}{\ov{p}}{X,U;\SS}\stackrel{\pr^*}{\to}
\lau \IH{i+1}{\ov{p}}{X;\SS}\to
\ldots,
$$
where $\pr \colon \lau \tN * {\ov p}{X,U}  \to \lau  \tN * {\ov p}{X} $ is defined by $\pr(\alpha,\beta) = \alpha$ and $\i \colon \lau \tN * {\ov p}{X} \rightarrow  \lau \tN * {\ov p}{U} $ is the restriction map
(cf. \cite[Sec. 12.2]{CST5}).

\medskip

e. {\bf  Injective limit.} 
Analogously to the Borel-Moore intersection homology, the    blown-up intersection cohomology with compact supports can be computed through the relative blown-up intersection cohomology by using an injective limit.
Let us see that.

\bp\label{limiteiny}
Let $
(X,\SS, \ov p)$ be a normal and hemicompact perverse space.
Then, there exists an isomorphism $$ \lau \IH * {\ov p,c} {X;\SS} \cong  \varinjlim_{K \subset X} \lau \IH *{\ov p} {X,X\menos K;\SS},$$
where $K$ runs over the family of compact subsets of $X$.
\ep
\bpr
By hemicompactness there exists an increasing sequence of compact subsets $\{K_n\}$ with
 $$
K_0 \subset \int{K_1} \subset K_1 \subset \int{K_2} \subset K_2 \subset \cdots K_n \subset \cdots , 
 $$
and $X = \bigcup_{n\geq 0} K_n$. In particular, the family $\{K_n\}$ is cofinal in the family of compact subsets of $X$. 
 So, it suffices to prove that the chain map
$$
B \colon \lau \tN *{\ov p,c} {X;\SS}\to \varinjlim_{n \in \N} \lau \tN *{\ov p} {X,X\menos K_n;\SS},
$$
defined by $B(\omega) =  \left<(\omega,0), {m} \right>$,
where $K_m$ is a compact support of $\omega$,
is a quasi-isomorphism.

An element $\left<(\alpha,\beta),m\right> \in \varinjlim_{n \in \N} \lau \tN *{\ov p} {X,X\menos K_n;\SS}$ is characterized by these two properties: 
\begin{itemize}
\item[-] $(\alpha,\beta) \in \lau  \IH *{\ov p} {X,X\menos K_m;\SS}$, and 
\item[-]  $\left<(\alpha,\beta),m\right> = \left<(\alpha',\beta'),m'\right>$  if  $ (\alpha,\beta) = (\alpha',\beta')$ on $\lau \tN *{\ov p} {X,X\menos K_{m'};\SS}$ if $m\leq m'$.
\end{itemize}
We proceed in several steps.

\medskip

$\bullet$ \emph{Step 1. Bump functions.}

\smallskip

Since $X$ is normal then,  for each each $n \in \N$, we can find a continuous map $f_n \colon X \to [0,1]$ with $f_n \equiv 0$ on $K_{n+1} $ and $f_n \equiv 1 $ on $X \menos \int{K_{n+2}}$. Associated to $f_n$ we have constructed a cochain $\tilde f_n \in \lau \tN 0 {\ov 0} {X;\mathcal S} $ verifying
$\tilde f_n \equiv 0 $ on $K_{n+1}$   and $\tilde f_n \equiv 1 $ on   $X \menos \int{K_{n+2}}$ (cf. \cite[Lemma 10.2]{CST5}). 
Consider the open covering $\mathcal U_n =\{ X\menos K_n,\int{K_{n+1}}\}$ of $X$.  Notice that\footnote{We refer the reader to (cf. \cite[Section 4]{CST5}) for the definition of the $\smile$-product.}, 
\be\label{fbeta}
 \gamma \in \lau \tN * {\ov p} {X\menos K_n;\mathcal S}
 \Longrightarrow 
  \tilde f_n  \smile \gamma \in \lau \tN {*,\mathcal U_n} {\ov p} {X;\mathcal S}
 \hbox { and }
 \tilde f_n  \smile \gamma  = \gamma
 \hbox{ on } X\menos{K_{n+3}}.
 \ee
 
 \medskip

$\bullet$ \emph{Step 2. The operator $B^*$ is a monomorphism.}

Let $[\omega] \in \ker B^*$. 
So, there exists $m\in \N$ and $\left<(\gamma,\eta),m\right> \in \lau \tN *{\ov p} {X,X\menos K_m;\SS}$ with $K_m$ compact support of $\omega \in \tN^*_{\ov p,c}(X;\SS)$ and $\left<(\omega,0),m\right> = \left<D(\gamma,\eta),m\right> = \left<(d\gamma, \gamma-d \eta),m\right>$. In particular, we get $\omega=d \gamma$ on $X$ and $\gamma-d\eta =0$ on $K\menos K_m$.

We get the claim if we prove that $\rho_{\mathcal U _m,c}(\omega) = d \theta $ for some $\theta \in \lau \tN{*,\mathcal U _m}{\ov p,c}{X;\SS}$ (cf. \cite[Theorem B]{CST5}). 
Here, 
$
\rho_{\mathcal U _m,c} \colon  \lau \tN{*}{\ov p,c}{X;\SS} \to\lau \tN{*,\mathcal U _m}{\ov p,c}{X;\SS}
$
is the canonical restriction.
 It suffices to consider  $\theta =\rho_{\mathcal U _m}( \gamma) - d(\tilde {f_m} \smile \eta)$,
 where $
\rho_{\mathcal U _m} \colon  \lau \tN{*}{\ov p,c}{X;\SS} \to\lau \tN{*,\mathcal U _m}{\ov p}{X;\SS}
$
is the canonical restriction, since
\begin{enumerate}[(i)]
\item $\tilde {f_m} \smile \eta\in \lau \tN{*,\mathcal U _m}{\ov p}{X;\SS}$ (cf. \eqref{fbeta}).
\item$\gamma - d (\tilde {f_m} \smile  \eta) \stackrel{\eqref{fbeta}} = \gamma - d\eta = 0$ on $X \menos K_{m+3}$, giving that $K_{m+3}$ is a compact support of $\theta$,
\item $d \theta = d \rho_{\mathcal U _m}( \gamma)  = \rho_{\mathcal U _m}( \omega) = \rho_{\mathcal U _m,c}( \omega)  $.
\end{enumerate}

\medskip

\medskip

$\bullet$ \emph{Step 3. The operator $B^{\mathcal U_m}$.}

For each $n,m \in \N$ we define
 $
  \lau {\tN}{*,\mathcal U  _m}{\ov p}{X,X\menos K_n;\SS} =
   \lau {\tN} {*,\mathcal U  _m}  {\ov p}{X;\SS}\oplus   \lau {\tN} {*-1,\mathcal U  _{m}} {\ov p}{X\menos K_n;\SS}$ (cf. \cite[Definition 9.6]{CST5}). We consider the chain map
$$
B^{\mathcal U  _m}  \colon \lau \tN {*,\mathcal U  _m} {\ov p,c} {X;\SS}\to \varinjlim_{n \in \N} \lau \tN {*,\mathcal U  _m} {\ov p} {X,X\menos K_n;\SS},
$$
defined by $B^{\mathcal U  _m} (\omega) = \left<(\omega,0),p\right>_{\mathcal U  _m}$,
where $K_p$ is a compact support of $\omega$ (cf. \cite[Definition 2.6]{CST7}).
We have the commutative diagram
\begin{equation}\label{diag:B}
\xymatrix@R=.5cm{
\ar[r]^-{B}  \lau \tN {*} {\ov p,c} {X;\SS}  \ar[d]^{\rho_{\mathcal U  _m,c} }& \varinjlim_{n\in\N} \lau \tN {*} {\ov p} {X,X\menos K_n;\SS}
\ar[d]^{\rho'_{\mathcal U  _m}}\\
\ar[r]^-{B^{\mathcal U  _m}} \lau  \tN {*,\mathcal U  _m} {\ov p,c} {X;\SS} & \varinjlim_{n\in\N}  \lau \tN {*,\mathcal U  _m} {\ov p} {X,X\menos K_n;\SS}
}
\end{equation}
 where the vertical maps are defined by restriction. Both are quasi-isomorphisms. It suffices to apply  \cite[Proposition 2.6]{CST7} (for the left one) and  the fact that inductive limits commute with cohomology and \cite[Theorem B]{CST5} (for the right one).

\medskip

$\bullet$ \emph{Step 4. The operator $B^*$ is an epimorphism.}

 Let $\Xi= \left<(\gamma,\eta),m\right>  \in \varinjlim_{\N} \lau \tN*{\ov p} {X,X\menos K;\SS}$ be a cycle. Then  $\left<D(\gamma,\eta),m\right>=\left<(d\gamma, \gamma- d\eta),m\right> =0$. 
The cochain  $\theta =\rho_{\mathcal U  _m}( \gamma) - d(\tilde {f}_m \smile \eta)$ is a cycle of $\lau \tN {*,\mathcal U  _m} {\ov p,c}{X;\SS}$ since (i), (ii) and 
$d \theta = d \rho_{\mathcal U  _m}( \gamma)  = 0$.
In fact,
\begin{eqnarray*}
B^{\mathcal U  _m,*}[\theta]&=& [\left<(\theta,0),{m+3}\right>_{\mathcal U  _m}] =
 \left[\left<( \rho_{\mathcal U  _m}(( \gamma) - d(\tilde {f}_m \smile \eta),0), {m+3}\right>_{\mathcal U  _m}\right] \\
&=_{(1)}& \left[\left<( \rho_{\mathcal U  _m}( \gamma), \tilde {f}_m \smile \eta),{m+3}\right>_{\mathcal U  _m}\right] 
=_{(2)} 
 [\left<( \rho_{\mathcal U  _m}( \gamma),  \eta),{m+3} \right>_{\mathcal U  _m}] 
 \\
 &=&
[\rho'_{\mathcal U  _m}(\left<(  \gamma,  \eta),{m+3}\right>)]
=_{(3)}  [\rho'_{\mathcal U  _m}( \left<( \gamma,  \eta),m\right>)] =  \rho'^*_{\mathcal U  _m}[\Xi].
\end{eqnarray*}
where $=_{(1)}$ comes from $D(\tilde {f}_m \smile \eta,0) = (d(\tilde {f}_m \smile \eta),-\tilde {f}_m \smile \eta)$, 
$=_{(2)}$  from $\tilde {f_m} = 1 $ on $X \menos {K}_{m+2}$ and $=_{(3)}$ from the fact that $\eta \in \tN^*_{\ov p}(X\menos K_m)$.

The properties of the previous diagram \eqref{diag:B} give the existence of $[\omega] \in\lau  \crH*{\ov p,c} {X;\SS}$ with
$
\rho^*_{\mathcal U  _m,c}[\omega] = [\theta]$  verifying
$
\rho'^*_{\mathcal U  _m}(B^{*}[\omega])  =B^{\mathcal U  _m,*}(\rho^*_{\mathcal U  _m,c}[\omega]) = B^{\mathcal U  _m,*}[\theta] =  \rho'^*_{\mathcal U  _m}[\Xi],
$
which gives $[\Xi] = B^{*}[\omega]$.
\end{proof}

\medskip

\section{Stratified sets and refinements}

\begin{quote}
A refinement of a stratified space $(X,\SS)$ is another stratified space $(X,\T)$ whose strata are formed using the strata of the original stratification.
We prove that it is possible to go from $\SS$ to $\T$ by modifying just a discrete family of strata: the simple refinement.
\end{quote}

\subsection{Stratified spaces}
A \emph{stratified space}\footnote{This definition is not a standard one in all sources. For example, it is more restrictive than that of \cite{CST1,LibroGreg}.} 
 is a  Hausdorff topological space $X$ endowed with a partition $\SS$, whose elements are called \emph{strata},  verifying the following conditions (S1)-(S6).

\begin{enumerate}[(S1)]

\item The family  $\SS$ is locally finite.

\item  An element of $\SS$ is a connected manifold.

\item \emph{Frontier Condition}. Given two strata $S,S' \in \SS$, we have\footnote{This condition is equivalent to say that the closure of a stratum is the union of strata.}: \hfill
$  S\cap \overline{S'}\neq \emptyset \Longrightarrow S\subset \overline{S'}.
$

\item  Given two strata $S,S' \in \SS$, we have: \hfill $S \cap  \ov{S' } \ne \emptyset  \hbox{ and } S \ne S' \Longrightarrow \dim S < \dim S'$.

\item The family $\{ \dim S \in  \SS\}$ is bounded.
\end{enumerate}

\smallskip

%
%
%A stratum is a \emph{regular stratum} if it is an open subset of $X$ or, equivalently, it is a maximal stratum of $(\SS,\preceq)$. A no regular stratum is a \emph{singular satratum}. The family of singular strata is denoted by $\SSsing$. We also write $\Sigma_\SS = \sqcup \{ S \in \SSsing\}$, or simply $\Sigma$, the \emph{singular part } of $X$. Notice that $X\menos \Sigma$ is a dense open subset of $X$.
%The \emph{depth of a stratum $S \in \SS$} is $d(S) = \max \{ i \in \N \mid \exists S_0 \prec S_1 \prec \cdots \prec S_i =S \hbox{ with } S_0, \ldots , S_i \in \SS'\}.$The \emph{depth of a 
%  family of strata $\SS' \subset \SS$} is 
%$\depth \SS' = \max \{d(S) \mid S \in \SS'\}.$
%In this work we shall use the formula 

Stratified and filtered spaces are related as follows.

\begin{lemma}
Let $(X,\SS)$ be a stratified space. Then 
the filtration $\emptyset = X_{-1} \subseteq X_0\subseteq X_1\subseteq \ldots \subseteq X_{n-1} \subsetneq X_n=X$, given by
$$
X_k = \sqcup \{ S \in \SS \mid \dim S \leq k\},
$$
  (cd. (S5)) defines a filtered space on $X$ whose associated stratification is $\SS$.
\end{lemma}
\begin{proof}
For the first statement it suffices to prove that each $X_k$ is a closed subset of $X$.  Let us consider $x \in \ov {X_k}$. Property (S1) gives $x \in \ov {X_k} = \ov{ \{ S \in \SS \mid \dim S \leq k\}} =
  \{ \ov S \in \SS \mid \dim S \leq k\}$. So, there exists $S \in \SS$ with $x \in \ov S$ and $\dim S \leq k$. If $S'$ is  the stratum of $\SS$ containing $x$ then condition (S4) give $\dim S' \leq \dim S$ and therefore $x \in X_k$.
  
  We have $X_k\menos X_{k-1} = \sqcup \{ S \in \SS \mid \dim S = k\}$. Again, conditions (S1) and (S4) imply that the elements of the RHS of the equality are closed subsets of $X_k\menos X_{k-1} $. So, the stratification of the filtered space is $\SS$.
\end{proof}

These are not equivalent notions since, for example, the strata of a filtered space are not necessarily manifolds.

The relation $S \preceq S'$ defined by $S\subset \overline{S'}$, is an order relation on $\SS$ (see \cite[Proposition A.22]{CST1}). The notation $S \prec S'$ means $S \preceq S'$ and $S \ne S'$. So, condition (S4) is equivalent to
\begin{itemize}

\item[(S4)] $S \prec S' \Longrightarrow \dim S < \dim S'$.

\end{itemize}

The \emph{depth}  
 of a  family of strata $\SS' \subset \SS$ is 
$\depth \SS' = \sup \{ i \in \N \mid \exists S_0 \prec S_1 \prec \cdots \prec S_i  \hbox{ with } S_0, \ldots , S_i \in \SS'\}.$ Conditions (S4) and (S5) give
\begin{lemma}\label{lem:fin}
 Let $(X,\SS)$ be  a stratified space. Then $ \depth \SS < \infty$.\end{lemma}

In this work we shall use the formula 

\be\label{adh}
\ov S = S \sqcup \bigsqcup_{Q\prec S } Q.
\ee
Let us see that. The inclusion $\supset$ is clear. Let $x\in \ov S$. Consider $Q\in \SS $ containing $x$. Since $Q \cap \ov\SS \ne \emptyset$ then we get $\subset$ from (S3).

 The examples of  \rojo \exemref{Ejemplos}\negro: induced, product \rojo with a manifold \negro, cone and join stratification \rojo with a manifold\negro, are also stratified spaces, if we begin with a stratified space $(X,\SS)$.
Any point  $x \in X$, with $\{x \} \not\in \SS$,  can be added as a new stratum. This gives the stratification 
\begin{equation}\label{ptplus}\rojo\SS_x =  \{ x \} \sqcup \{ (S\menos \{x\})_{cc}\mid S \in \SS\}.
\end{equation}

\negro

\subsection{Refinements}
 We say that the stratified space $(X, \SS)$ is a \emph{refinement} of  the stratified space $(X, \mathcal T)$\footnote{We also say that $(X,\mathcal T)$ is a \emph{coarsening} of  $(X,\SS)$.}, written $(X,\SS) \triangleleft (X,\mathcal T)$,  if $\SS \ne \mathcal T$ and
\begin{itemize}
\item[(S6)] \rojo  any stratum $ T \in \mathcal T$ is a union of strata of $\SS$.\negro
\end{itemize}
 \rojo When  $S\in \SS, T \in \mathcal T$ and $S\subset T$ we also write $T =\bi SI$ \negro. We have
\begin{equation}\label{eq:desing}
\dim  S\leq \dim\bi  SI  \hbox{ and  }  \codim \bi SI \leq \codim S, \ \ \ \hbox{for each } S \in \SS.
\end{equation}
Notice that
\begin{equation}\label{SQSIQI}
S, Q \in \SS \hbox{ with } S \preceq Q \Longrightarrow \bi SI \preceq \bi QI
\end{equation}
(cf. (S3)).

\begin{remarque}\label{RemS6}
{\rm 
This condition (S6) is weaker that the  (S6) condition of \cite{SRef1}. In both cases, any stratum $T\in \TT$ is union of strata of $\SS$ but the condition (S6) of \cite{SRef1} asks for a particular local structure. More precisely, given a point $x\in X$ belonging to the strata $S \in \SS$ and $T \in \TT$  we locally have 
$
\left\{
\begin{array}
{lll}
w&=&(0,\tv)\\
S &=& \R^a \times \{ \tv\}\\
T&=& \R^a \times \rc S^{b-1},
\end{array}
\right.
$
where $S^{b-1}$ is the $(b-1)$-dimensional sphere and $\tv$ the apex of the cone.
See \remref{New} for a comparison between the two (S6) conditions.

In fact, this extra condition appears on the local calculations of \secref{sec:CS} and not on the other Sections.
}
\end{remarque}

In this work, we shall distinguish  several types of strata.

\bd\label{def:AOVM}
Let $(X,\SS) \triangleleft (X,\mathcal T)$ be a refinement. A stratum $S\in \SS$ is a \emph{source stratum }  if $\dim S = \dim \bi SI$. In this case, we also say that $S$ is a \emph{source stratum of $T \in \TT$}, if $T=\bi SI $. We also use the following types of strata:

\begin{itemize}
\item[-]  $\mathcal A =\{ S \in \SS \ / \ \dim S = \dim \bi SI\}$:  \hfill \emph{source strata}.

\item[-]  $\mathcal V =\{ S \in \SS \ / \ \dim S < \dim \bi SI\}$: \hfill \emph{virtual strata}.

\item[-]  $\mathcal M =\{ \hbox{maximal strata of } \mathcal V \hbox{ with $\dim \bi MI$ maximal}\}$: \hfill \emph{v-maximal strata}.

\item[-]  $\mathcal O =\{ S \in \mathcal A \ / \  \exists M \in \mathcal M \hbox{  with } M \preceq S , \bi MI = \bi SI  \}$:\hfill  \emph{stable strata}.

\end{itemize}

The refinement $(X,\SS) \triangleleft (X,\mathcal T)$ is \emph{simple} when  $\depth \mathcal V=0$.
 If $V \in \mathcal V$ and $Q \in \SS$ with $V \prec Q$ then $Q \in \mathcal A$.
We always have $\depth \mathcal M=0$.
\ed

\begin{remarque}\label{S=T}{\rm
By dimensional reasons, for any stratum $T \in \TT$, the union of source strata of $\SS$ included in $T$ is an open dense subset of $T$ and therefore a $\ell$-dimensional manifold for $\ell =\dim T$ (cf. (S1), (S2) and (S6)).
By connectedness, the stratum  $T$ can not be the union of a family of source strata except when $T \in \SS$.
So:
$
\SS \ne \TT \Longrightarrow \mathcal V \ne \emptyset \Longrightarrow \mathcal M \ne \emptyset .
$
}
\end{remarque}

\bd
Let $(X,\SS) \triangleleft  (X,\T)$ be a refinement. A stratum $S \in \SSsing$ is \emph{exceptional} if  $\bi SI \in \TTreg$. Moreover, if $\codim S=1$ we say that $S$ is an \emph{1-exceptional stratum}.
\ed

Any exceptional stratum is a virtual stratum.

\begin{definition}
A \emph{simple decomposition} of a refinement 
 $(X,\SS) \triangleleft (X,\mathcal T)$ is a finite sequence of 
 simple refinements: 
$
(X,\SS)  = (X,\mathcal R_0)\triangleleft  \cdots \triangleleft  (X,\mathcal R_m) =(X,\mathcal T).
$
\end{definition}

\begin{example}\label{ejemplo}
{\rm
A key result of this work is the \propref{111} giving the existence of  simple refinements. The relevance of these kind of refinements is given by \propref{Max}, where we get a nice local description of the a simple refinement  in the framework of CS-sets.

Before proving these results, we give an example of a refinement $(X,\SS) \triangleleft_I (X,\mathcal T)$ described as composition of two simple refinements $(X,\SS) \triangleleft_J (X,\mathcal S')$  and $(X,\SS') \triangleleft_K (X,\mathcal T)$  through a stratified space $(X,\SS')$.

\bigskip

\begin{center}
\begin{tikzpicture}[scale=0.9]

%%%%%%% (X,S)%%%%%%%%%

\draw[dashed](0,0)-- (3,0);
\draw[dashed] (0,0)-- (0,3);
\draw(.5,1.5)--(2.5,1.5);
\draw(1.5,0.5)--(1.5,1.5);

\draw(1.5,1.5) node {$\bullet$} ;
\draw(0.7,2.2) node {$\bullet$} ;
\draw (0.7,2.5) node {$Q_2$} ;
\draw (2.7,2.5) node {$Q_3$} ;
\draw(2.7,2.2) node {$\bullet$} ;
\draw (1.7,1.9) node {$Q_1$} ;
\draw (1.7,2.7) node {$R_1$} ;
\draw (1.5,0.3) node {$S_3$} ;
\draw (0.3,1.5) node {$S_1$} ;
\draw (2.7,1.5) node {$S_2$} ;
\draw (1.5,-.5) node {$ (X,\SS)$} ;
\draw (0.3,0.3) node {$R_3$} ;
%\draw (0.3,2.7) node {$R_3$} ;
%\draw (2.7,2.7) node {$R_2$} ;
\draw (2.7,0.3) node {$R_2$} ;

\draw[->] (1.5,-1) -- (4.5,-2);
\draw (3,-1.2) node {$J$} ;

\draw[->] (3.5,1.7) -- (9.5,1.7);
\draw (6.5,2) node {$I$} ;

\draw (1.5,-4.55) node {Refinement $J$} ;
\draw (1.5,-5) node {$\mathcal A \menos \mathcal O= \{Q_1,Q_2,Q_3,R_1,S_1,S_2\}$} ;
\draw (1.5,-5.5) node {$\mathcal V =\mathcal M = \{S_3\}$} ;
\draw (1.5,-6) node {$\mathcal O = \{R_2,R_3\}$} ;

%%%%%%% (X,S')%%%%%%%%%

\draw[dashed](5,-3)-- (8,-3);
\draw[dashed] (5,-3)-- (5,0);
\draw(5.5,-1.5)--(7.5,-1.5);
\draw (5.3,-1.5) node {$S_1$} ;
\draw (7.7,-1.5) node {$S_2$} ;
\draw (6.7,-1.1) node {$Q_1$} ;
\draw (6.7,-2.7) node {$R_4$} ;
\draw (5.7,-.5) node {$Q_2$} ;
\draw (7.7,-.5) node {$Q_3$} ;
\draw (6.7,-0.3) node {$R_1$} ;
\draw(7.7,-.8) node {$\bullet$} ;
\draw (6.5,-3.5) node {$ (X,\SS')$} ;
\draw(6.5,-1.5) node {$\bullet$} ;
\draw(5.7,-.8) node {$\bullet$} ;

\draw[->] (8.2,-2) -- (11.5,-.8);
\draw (9.8,-1.1) node {$K$} ;

\draw (11.5,-.5) node {$ (X,\mathcal T)$} ;

\draw (6.5,-4.5) node {Refinement $I$} ;
\draw (6.5,-5) node {$\mathcal A \menos \mathcal O= \{Q_3,R_1,S_1,S_2\}$} ;
\draw (6.5,-5.5) node {$\mathcal V \menos \mathcal M= \{Q_1,Q_2\}$} ;
\draw (6.5,-6) node {$ \mathcal M = \{S_3\}$} ;
\draw (6.5,-6.5) node {$\mathcal O = \{R_2,R_3\}$} ;

%%%%%%% (X,S'')%%%%%%%%%

\draw[dashed](10,0)-- (13,0);
\draw[dashed] (10,0)-- (10,3);

\draw (12.7,1.5) node {$S_4$} ;
\draw (11.7,0.3) node {$R_4$} ;
\draw (11.7,2.7) node {$R_5$} ;
\draw(10.5,1.5)--(12.5,1.5);
\draw (11.5,-.5) node {$ (X,\mathcal T)$} ;
\draw (12.7,2.7) node {$Q_3$} ;
\draw(12.7,2.4) node {$\bullet$} ;

\draw (11.5,-4.5) node {Refinement $K$} ;
\draw (11.5,-5) node {$\mathcal A \menos \mathcal O= \{Q_3,R_4\}$} ;
\draw (11.5,-5.5) node {$\mathcal V =\mathcal M = \{Q_1,Q_2\}$} ;
\draw (11.5,-6) node {$\mathcal O = \{R_1,S_1,S_2\}$} ;

\end{tikzpicture}
\end{center}

In the simple refinement $J$ (resp. $K$) 
a stratum of $\mathcal M$ melts into a stratum of $\mathcal \SS'$ (resp. $\TT$) and, for each of them,  1 or 2 strata of $\mathcal O$ also disappear into a bigger stratum with same dimension: $S_3,R_2,R_3 \leadsto R_4$  (resp. $Q_1,S_1,S_2 \leadsto S$ or $Q_2,R_1 \leadsto R_5$). Among the strata of $\mathcal V$ those of $\mathcal M$ are the first to disappear.

}\end{example}

The objective of the following Lemmas is  to prove   \propref{111}: a refinement can be decomposed as a sequence of simple refinements.

\bl\label{refprop}
Let $(X,\SS) \triangleleft _I(X,\mathcal T)$ be a refinement with $\SS \ne \TT$. Then

\begin{enumerate}[(a)]
%
%\item
%$
%\mathcal V \ne \emptyset .
%$

%\item For each open subset $V \subset X$ we have the refinement  $(V,\SS) \preceq(V,\mathcal T)$.

\item  $\mathcal M \ne \emptyset$.

\item  For each $S \in \SS$ there exists a source stratum $P\in \SS$  of $\bi SI$ with $S \preceq P$.

\item Given two strata $R,Q \in \TT$ with $R\preceq Q$ there exist two source strata $R',Q' \in \SS$ of $R$ and $Q$ respectively with $R' \preceq Q'$.

\end{enumerate}
\el
\bpr (a) Since $\SS \ne \TT$ then there exists a stratum $S \in \SS$ with $S \ne \bi SI$. If $\mathcal M=\emptyset$ then $\mathcal V = \emptyset$ and then $\bi S I =  \sqcup \{P\in \mathcal A \mid \bi PI = \bi SI\}$, open connected subsets of $\bi SI$  (cf. (S6)). By connectedness of $\bi SI$ we conclude that $ \{P\in \mathcal A \mid \bi PI = \bi SI\}$ contains just one element, necessarily $S$. The contradiction $\bi SI=S$ implies that $\mathcal M \ne \emptyset$. 

(b) By definition  we have
\begin{equation}\label{eq:SIdescom}
\bi S I =  \sqcup \{P\in \mathcal A \mid \bi PI = \bi SI\} \sqcup \left( \sqcup \{P\in \mathcal V \mid \bi PI = \bi SI\}\right),
\end{equation}
where the elements of the first term are open subsets of $\bi S I$.
This decomposition is locally finite (cf. (S1)). By dimension reasons,   $ O =\sqcup\{P\in \mathcal A \mid \bi PI = \bi SI\} =
\sqcup \{ \hbox{source strata of } \bi SI\}$ is an open dense subset of $\bi SI$ (cf. \rojo (S1), \negro (S4), (S6)). Condition $S \subset \ov O $ implies the existence of a	 source stratum $P$ of $\bi SI$ with  $S \cap \ov P \ne \emptyset$. Property (S3)  gives (b).

(c) Item (b) gives a source stratum $R' \in  \SS$ of $R$.  Since 
$ \sqcup\hbox{ \{source strata of $Q$\} }= \sqcup_{i\in I} Q_i $
is an open dense subset of $Q$ then $
R' \subset {R}  \subset \ov Q = \ov {\sqcup_{i\in I} Q_i } =_{(S1)}  \cup_{i\in I} \ov{Q_i } $. So, there exists $Q_i \in \mathcal S$, source stratum of $Q$, with  $R' \cap \ov {Q_i} \ne \emptyset$. Since $R' \preceq {Q_i}$ (cf. (S3)) we end the proof  taking $Q' = Q_i$. \epr

%(c) Property (b) gives a source stratum $S' \in  \SS$ of $\bi SI$ with $S \preceq S'$. Since 
%$
%S' \subset \bi {S'}I =\bi SI \stackrel{\eqref{SQSIQI}} \subset \ov {\bi QI},
%$
%then there exists $R \in\SS$ with $ \bi RI = \bi QI$ and  $S' \preceq R$. Again, property (b) gives us a source stratum $Q'\in \SS$  of $\bi RI$ with  $R \preceq Q'$. This ends the proof.

The subsets $\dos MI$ we study now play an important r\^ole in the construction of the simple decomposition of a refinement. They are the new strata on the first step of this decomposition.
\bl\label{MR}
Let $(X,\SS) \triangleleft_I (X,\mathcal T)$ be a refinement. Consider a stratum  $M \in \mathcal M$.
We define
\begin{equation*}\label{eq:defMI}
\dos M I=\sqcup \{Q  \mid Q \in \mathcal O \sqcup \mathcal M \hbox{ and }\bi QI =  \bi MI\}.
\end{equation*}
Then 
\begin{enumerate}[(a)]
\item $\dos MI$ is a connected open subset of $\bi MI$.
\item \rojo $Q \subset \dos MI$, if $Q \in \mathcal M$ and $\bi QI =\bi MI$, and\negro
\item $Q$ is an open subset of $\dos MI$, if $Q \in \mathcal O$ and  $\bi QI =\bi MI$.
\end{enumerate}
\el

\bpr\label{MR2}
Without loss of generality we can suppose $X = \bi MI$. 
We have $\bi SI =\bi MI$ for each $S \in \SS$.

(a) The subset $F =\sqcup \{ S \in \mathcal V \menos \mathcal M\}$ is a closed subset of $X$ (cf. \eqref{adh} and (S4)) not meeting $\dos MI$. Given $S \in  \mathcal V \menos \mathcal M$ we have $ \dim S \ne \dim X$ (since $S \in \mathcal V$) and $\dim S \ne \dim X-1$ (since $S \not\in \mathcal M)$.
%$\dim$S \prec N \prec Q$ for some $N \in \mathcal M$ and $Q \in \mathcal A$  (cf. Lemmas  \ref{lem:fin} and \ref{refprop} (b)). 
Then, $F$ it is a locally finite union of sub-manifolds of $X$ whose codimension is at least 2 (cf. (S1), (S4)). So, $Y = X\menos F$ is a $\SS$-saturated connected open subset of $\bi MI$ containing $\dos MI$.

By construction, we have  $\mathcal V_Y = \mathcal M$, that is, $\SS_Y = \mathcal M \sqcup \mathcal A$. Let us suppose $\mathcal A \ne \mathcal O$. By dimension reasons, if $S \in \mathcal A \menos \mathcal O$ then $\ov S =S$ (cf. \eqref{adh} and (S4)). Property (S2) gives $S=Y$ and then $S=M$ which is impossible. So,  $\SS_Y = \mathcal M \sqcup \mathcal O$. Then  $\dos MI = Y$  which is a connected open subset of $\bi MI$.

(b) 
\rojo By definition. \negro

% Condition (S4) implies that the strata of $\mathcal M$ are minimal strata of $\SS_Y$, and therefore closed subsets of $Y$ (cf. \eqref{adh}).
%The subset $C =\sqcup \{ S \in \mathcal M \mid S \ne Q\}$ is a closed subset of $X$ since $\ov C =_{(S1)} \bigcup_{Q\ne S \in \mathcal M} \ov S =  \bigcup_{Q \ne S \in \mathcal M} S = C$. So, the subset $ Z= Y\menos C$ is an open subset of $Y$. In particular, we have  $\SS_Z = \{ Q \} \sqcup \mathcal O$ which gives $Q \subset Z \subset \dos MI $ and therefore (b).

(c) Finally, $\dim Q = \dim \bi QI = \dim \bi MI$ implies that $Q$ is an open subset of $\bi MI$ (cf. (S6)). Since $Q \subset \dos MI$, condition (a) gives (c).
\epr

%
%\bl\label{virt}
%Let $(X,\SS) \triangleleft (X,\mathcal T)$ be a refinement.
%Consider a stratum  $M \in \mathcal M$.
%There exists a saturated open subset $U \subset X$  containing $\dos MI$ and verifying:
% $
%V \in \mathcal V  \hbox{ with } V \cap U \ne \emptyset \Rightarrow V  \subset \dos MI.
% $
%\el
%\bpr
%%Let $V \in \mathcal V$ with  $\ov V \cap \dos MI \ne \emptyset$. Then there exits $R \in \mathcal M$ with $R \subset \dos MI$ and $R \preceq V$ (cf. (S3)$_{\SS}$ and \lemref{MR}).  This gives $V = R \subset \dos MI$. 
%Consider  $F=\bigsqcup \{ \ov V \mid  V \in \mathcal V, V \not\subset \dos MI\}$. It is a saturated closed subset (cf. (S1) and \eqref{adh}). We prove that $F$ not meet $\dos MI$. Suppose there exist 
%\begin{itemize}
%\item $Q \in \mathcal M \sqcup \mathcal O$ with $Q \subset \dos MI$, and
%\item $V \in \mathcal V$ with $Q \preceq V$ and $V \not\subset \dos MI$.
%\end{itemize}
%If $Q \in \mathcal M$ then $Q=V$ which is impossible. If $Q \in \mathcal O$ then there exists $N\in \mathcal M$ with $N \preceq Q$. This gives $N \preceq Q \preceq V$ and then $N=Q=V$ which is impossible.
%So, the saturated open subset $U = X \menos F$ contains $\dos MI$.
%Let $V \in \mathcal V$ be a stratum with $V \cap U \ne \emptyset$. Then $V \subset U = X\menos F$ and therefore $V \cap F= \emptyset
%$. This gives $V \subset \dos MI$.\epr
%

\bp\label{111}
Any refinement $(X,\SS) \triangleleft (X,\mathcal T)$, with $\SS \ne \TT$,  has a simple decomposition.
\ep

\bpr
Let us define $d_{\SS,\TT} = \dim \bi MI$ where  $M \in \mathcal M$. This number is independent of the choice of $M$ by definition of $\mathcal M$.
Condition $\SS \ne \TT$ implies  $\mathcal M\ne \emptyset$ (cf. \lemref{refprop} (a)) and therefore  $d_{\SS,\TT}\geq 0$.
We proceed by induction on $d_{\SS,\TT}$. If $d_{\SS,\TT}=0$ then the dimension of the strata of $\mathcal M$ is 0. Then $\mathcal V = \mathcal M$, which gives $\depth \mathcal V=0$. We conclude that the refinement is simple.

Now, in the inductive step,  we can suppose that $ d_{\SS,\mathcal T}>0$.
It suffices to construct a chain of refinements 
$
(X,\SS)  \preceq (X,\mathcal R)\preceq (X,\mathcal T),
$
where the first one is simple and 
$ d_{\mathcal R,\mathcal T} < d_{\SS,\mathcal T}$.

Let $M,N\in \mathcal M$ be two strata with $\dos MI \cap \dos NI \ne \emptyset$. This implies $\bi MI \cap \bi NI \ne \emptyset$
and therefore $\bi MI = \bi NI$. So, $\dos MI = \sqcup \{Q \in \mathcal O \sqcup \mathcal M \mid Q \subset \bi MI\} = \sqcup \{Q \in \mathcal O \sqcup \mathcal M \mid Q \subset \bi NI\} = \dos NI$.
We get the dichotomy $\dos MI = \dos NI$ or $\dos MI \cap \dos NI =\emptyset$. In order to avoid repetitions, we fix a family $\{ M_i \subset \mathcal M\mid i \in \nabla\}$ such that 
$
\cup \{ \dos MI \mid M \in \mathcal M\} = \sqcup \{ \dos M{i,I} \mid i \in \nabla\},
$
We define
\begin{equation}\label{R2}
\mathcal R = 
\mathcal S \menos (\mathcal O \sqcup  \mathcal M) \sqcup  \{ \dos M{i,I} \mid i \in \nabla\}.
\end{equation}
Let us verify all the properties.

\medskip

$\bullet$ {\bf {\boldmath $(X,\mathcal R)$} is a stratified space}. By definition of stable strata we have  $\sqcup \{ Q  \mid  Q\in \mathcal O \sqcup  \mathcal M\}  = \sqcup \{ \dos M{i,I} \mid i \in \nabla\}$. Then  $\mathcal R$ is a partition of $X$.
Condition (S1)$_\SS$ gives condition (S1)$_\mathcal R$. Condition (S2)$_\mathcal R$ comes from  (S2)$_\SS$ and \lemref{MR} (a).
 For the proof of (S3)$_\mathcal R$ and (S4)$_\mathcal R$, it suffices to prove:
 
 \smallskip
 
\begin{tabular}{lcl}
(a)  $S\cap \ov P \ne \emptyset$  & $\Rightarrow $ & $S \subset \ov P$ and $\dim S < \dim P$.\\
(b)   $S \cap \ov{\dos MI }\ne\emptyset$ & $\Rightarrow$ & $S \subset \ov { \dos MI }$ and $ \dim S < \dim \dos MI $,\\
(c)  $\ov S \cap \dos MI \ne \emptyset $ & $\Rightarrow$& $ \ov S \supset  \dos MI $ and $\dim \dos MI < \dim S$, \\
(d) $\ov{\dos NI} \cap {\dos MI }\ne\emptyset$ & $\Rightarrow$ &$\dos MI = \dos NI $.
\end{tabular}

\smallskip

\noindent where $S,P \in \mathcal S \menos (\mathcal O \sqcup  \mathcal M)$ and $M,N \in \mathcal M$. 
Let us see that.
\medskip

(a) It follows directly from (S3)$_\SS$ and (S4)$_\SS$.

\smallskip

(b) Locally finiteness of $\SS$ (cf. (S3)$_\SS$) gives
$\ov{\dos MI} =\cup \{\ov Q  \mid Q \in \SS \hbox{ and } Q \subset \dos MI\}$.
So, there exists  $Q\in \mathcal O \sqcup \mathcal M$ with $\bi QI =\bi  MI$ and  $S \cap \ov Q \ne \emptyset$. So,  $S \subset \ov Q \subset \ov {\dos MI}$ (cf. (S3)$_\SS$). 
Since $S \not\in \mathcal O \sqcup \mathcal M$ we get $S \ne Q$ and then
$\dim S
\stackrel{ (S4)_\SS} < \dim Q \stackrel{\eqref{eq:desing} }\leq \dim \bi QI = \dim \bi MI \stackrel{\lemref{MR} (a)}  = \dim \dos MI$.

%
%If $Q \in \mathcal M$ then $\dim S \leq \dim Q < \dim \dos MI$ while if   $Q \in \mathcal O$ then $S \in \mathcal V$ and therefore $\dim S < \dim Q \leq \dim \dos MI$  (cf. (S4)$_\SS$, \lemref{MR} (d)).

\smallskip

(c) Condition $\ov S \cap \dos MI \ne \emptyset$ implies the existence of $Q \in \mathcal M \sqcup \mathcal O$ with $\bi QI
=\bi MI$ and $Q\preceq S$ (cf. (S3)$_\SS$).
By definition of stable strata we can suppose that $Q \in \mathcal M$, which implies $S \in \mathcal A$ since $S \not\in \mathcal M$. If $\bi MI =\bi SI$ then $S \in \mathcal O$, which is impossible. So, $\bi MI \ne \bi SI$. Since $\bi MI = \bi QI \stackrel{\eqref{SQSIQI}} \preceq \bi SI$  then $\bi MI  \prec \bi SI$ and we get  $\dim \bi MI  \stackrel{(S4)_\TT} \prec \dim \bi SI$.

Let us consider a virtual stratum $V \in \mathcal V$ included in $\bi SI$. There exists a maximal stratum $W \in \mathcal V$ with $V \preceq W$ (cf. \lemref{lem:fin}). Since $\bi VI \preceq \bi WI$ (cf. \eqref{SQSIQI}) then we have
$$
\dim \bi MI < \dim \bi SI = \dim \bi VI \stackrel{(S4)_{\mathcal S}} \leq \dim \bi WI,
$$
which is impossible by definition of $\mathcal M$. So,  the subset $\bi SI$ does not contain any virtual stratum. 

By connectedness of $\bi SI$ the formula \eqref{eq:SIdescom} implies that $\bi SI$ contains just one stratum of $\SS$, that is, $\bi SI = S$. We get $\dos MI \subset \bi MI = \bi QI \subset \ov{\bi SI} = \ov S$ and $\dim \dos MI = \dim \bi MI = \dim\bi QI \stackrel{(S4)_{\mathcal S}}   \leq \dim \bi SI = \dim S$ (cf. \lemref{MR} (a)).

\smallskip

(d) If $\ov{\dos NI }\cap \dos MI \ne \emptyset$ then $\ov{\bi NI }\cap \bi MI \ne \emptyset$ and therefore $\bi MI \preceq \bi NI $ (cf. (S3)$_{\TT}$). \lemref{MR} (a), (S4)$_\TT$ and \eqref{SQSIQI} give
$
\dim \dos MI = \dim \bi MI \leq \dim \bi NI = \dim \dos NI.
$
By definition of $\mathcal M$ we get that previous $\leq $ becomes $=$. Finally, condition (S4)$_\TT$ gives $\bi MI = \bi NI$ and therefore $\dos MI = \dos NI$.

\medskip

$\bullet$ {\bf \boldmath{$(X,\SS)  \triangleleft  (X,\mathcal R)$} is a simple refinement}.  
The strata of $\SS \menos (\mathcal M \sqcup \mathcal O)$ remain equal. The other strata verify condition (S6)$_{\SS,\mathcal R}$ following  \lemref{MR}. So, $(X,\mathcal S)  \triangleleft (X,\mathcal R)$ is a refinement.
The only strata whose dimension increases when passing from $\SS$ to $\RR$ are the strata of $\mathcal M$: $\dim M < \dim \dos MI$. So 
\begin{equation}\label{eq:MV}
\mathcal V_{\SS,\RR}= \mathcal M_{\SS,\RR}  = \mathcal M = \mathcal M_{\SS,\TT}
\end{equation}
 which gives $\depth \mathcal V_{\SS,\mathcal R} = \depth \mathcal M_{\SS,\TT} =0.
$

\medskip

$\bullet$  {\bf {\boldmath$(X,\mathcal R)  \triangleleft  (X,\mathcal T)$} is a  refinement with {\boldmath $d_{\mathcal R,\mathcal T} < d_{\SS,\mathcal T}$}}. 
\rojo A stratum $Q \in \SS \menos (\mathcal O \sqcup \mathcal M)$ belongs to $\mathcal R$ and is included  in $\bi QI$ from (S6)$_{\SS,\mathcal T}$.  \negro The strata $\dos MI$, $M \in \mathcal M$, are open subsets of $\bi MI$. So, $(X,\mathcal R)  \triangleleft (X,\mathcal T)$ is a refinement.
Since $\dim \dos MI = \dim \bi MI$, for each $M \in \mathcal M$, then $\dos MI \in \mathcal R$ is a source stratum. 
The same is true for the strata of $\mathcal A \menos \mathcal O$. This gives  $ \mathcal V_{\RR,\T} = \mathcal V \menos \mathcal M
=
\mathcal V_{\mathcal S,\mathcal T} \menos \mathcal M_{\mathcal S,\mathcal T}$
and therefore  $ d_{\mathcal R,\mathcal T}  < 
 d_{\mathcal S,\mathcal T}$. \epr

%
%\bl\label{25}
%Let $(X,\SS) \triangleleft (X,\mathcal T)$ be a refinement. Given two strata $T \preceq T'$ of $\T$ then there exist two source strata $S,S'$ of $T,T'$, respectively,  with $S \preceq S'$.
%\el
%\bpr 
%From \lemref{refprop} (c), there exist two source strata $S,Q$ of $T$, $T'$ respectively. Condition $\bi SI  =T\preceq T'=\bi {Q}I$ implies the existence of a stratum $P\in \SS$ with $S \preceq P$  and $P \subset T'$ (cf. (S3)).
%Again, \lemref{refprop} (c) gives a source stratum $S' $ with $P \preceq S'$. Son  $S \preceq  S'$.
%% Without loss of generality we can suppose $R \in \mathcal V$ and $R' \in \mathcal O$ (cf. \lemref{MR}). Since 
%%$\bi PI \cap {R'} =\emptyset$
%% and $\bi PI \cap \ov {R'} \ne \emptyset$ then 
%%$\bi PI \subset \ov{R'} $ (cf. Note Bene of the proof of \lemref{11} (b)).
%%It suffices to take $S=P$ and $S'=R'$.
%\epr

\section{CS-sets} \label{sec:CS}

\begin{quote}
%:
%:
%:
The invariance result we study in this work applies to CS-sets, a weaker notion than that of stratified pseudomanifold. Here, a link  of a stratum is not necessarily a CS-set  but a filtered space \cite[example 2.3.6]{LibroGreg}.  We also describe the local structure of a simple refinement between two CS-sets. \end{quote}

\subsection{CS-sets}
 A filtered space $(X,\SS)$ is a \emph{$n$-dimensional CS-set} if
any regular stratum is an $n$-dimensional manifold,
and for any singular stratum $S \in \SS$ and for any $x \in S$ there exists a \rojo stratified homeomorphism\negro\footnote{The involved  stratifications are described in \exemref{Ejemplos} .}
$$
\varphi \colon (\R^i \times \rc L, \I \times \tc \L)  \to (V,\SS),
$$
where
\begin{enumerate}[(a)]
\item $V \subset X$ is a open subset containing $x$,
\item $(L;\L)$ is a compact filtered  space,
\item $\varphi(0,\tv) = x$ and $\varphi (\R^i \times \{ \tv \} ) = V \cap S$.

\end{enumerate}

The pair  $(V,\varphi)$ is a   $\SS$-\emph{conical chart}, or simply \emph{conical chart},  of $x$.
The \emph{link} of $\varphi$ is $(L,\L)$. 
Since the links are always  non-empty sets then the open subset $ X \backslash \Sigma $ is dense.
Closed strata of $\SS$ are exactly the minimal strata of $\SS$. On the other hand, the open strata of $\SS$ are the maximal strata of $X$, they coincide with the $n$-dimensional strata of $X$.

A {\em perverse CS-set} is a triple $(X,\SS,\ov p)$ where $(X,\SS)$ is a CS-set and $\ov p$ is a perversity on $(X,\SS)$. 

%In order to have a global setting we shall say a \emph{conical chart} of a regular point $x \in X\menos  \Sigma_X$ is a a pair $(V,\varphi)$ where $\varphi \colon \R^n \to V$ is a homeomorphism where $V \subset X \menos \Sigma_X$ is an open neighborhood of $x$. We shall say that $(\emptyset,\emptyset)$ is the \emph{link } of $x$.

%
%A CS-set is a {\em pseudomanifold} if the link is a pseudomanifold. Moreover, if the pseudomanifold $X$ does not possesses any one-codimensional strata, we find the notion of \emph{classical pseudomanifold}.

We find in \cite{MR2654596} a comparison between different notions of stratification. In this work we need the following property.

\bp
Any CS-set is a stratified space.
\ep
\bpr
Conditions (S2) and (S5) come  from definition. Property (S1) is proved in \cite[Lemma 2.3.8]{LibroGreg}. Let us verify (S3) and (S4). Since it is a local question, we  set $X  = \R^i \times \rc L$ with $S =  \R^i \times \{ \tv\}$. We can suppose  $S \ne S'$ and therefore $S' = \R^i \times Q \times ]0,1[$ for some $Q \in \mathcal L$. Since $\ov {S'} = \R^i \times \tc \ov Q$ we get $S \subset \ov {S'}$. We also have $\dim S < \dim S'$.
\epr

Consider a refinement  $(X,\SS) \triangleleft (X,\T)$ between two CS-sets, which makes sense following previous Proposition.
 The identity $I \colon (X,\SS) \to (X,\T)$ is in fact a stratified map (cf. \eqref{eq:desing}). We write $(X,\SS) \triangleleft _I (X,\T)$.
 
 \smallskip
 
Simple decompositions and CS-sets are compatible.

\bp\label{12}
A refinement $(X,\SS) \triangleleft (X,\mathcal T)$ between two different  CS-sets possesses a simple decomposition made up of  CS-sets.
\ep
\bpr
 It suffices to prove that the first element $(X,\mathcal R)$ of  the simple decomposition constructed in the proof of \propref{111} is a CS-set.

We use the following notation:
$(X,\mathcal S) {\triangleleft}_I (X,\mathcal R)
{\triangleleft}_J(X,\mathcal T)$
and
$(X,\mathcal S){\triangleleft}_E (X,\mathcal T)$ the original refinement.
We know that the manifolds $X\menos \Sigma_\SS$ and $X\menos \Sigma_\T$ are dense open subsets of $X$. So, $\dim (X,\SS) = \dim (X,\T)$.

It remains to construct a $\mathcal R$-conical chart of any point $x \in \Sigma_\SS$.
We consider the strata $S \in \SS$ and $\bi SI \in \mathcal R$ containing $x$. We distinguish two cases.

\smallskip

{\boldmath+  $S \in \mathcal A_{\SS,\mathcal R} $.}
Let  $\varphi \colon (\R^m \times \rc L, \I \times \rc \L) \to (V,\SS)$ be a $\SS$-conical chart  of $x$ with  link  $(L,\L)$. Since $\dim S =\dim \bi SI $ then 
 $
S\cap V =  \bi SI \cap V  = \varphi (\R^m \times \{\tv\}).
$
A  stratum of 
 $\mathcal R_{V\menos \bi SI }$ is a union of strata of $\SS_{V\menos S}$, then it is of the form  $\varphi( \R^m \times ]0,1[ \times \bullet)$. So, there exists a filtration $\L'$ on $L$ such that $\varphi \colon  (\R^m \times ]0,1[ \times L, \I \times\I \times \L') \to (V\menos \bi SI , \mathcal R)$ is a stratified homeomorphism. This is also the case for $\varphi \colon  (\R^m  \times\rc L, \I \times \rc\L') \to (V, \mathcal R)$.
We get that   $(\varphi, V) $ is a   $\mathcal R$-conical neighborhood of $x$ with link $(L,\L')$.

\smallskip

 {\boldmath + $S \in \mathcal V_{\SS,\mathcal R}$}. 
 Notice first that $\mathcal V_{\SS,\mathcal R} =  \mathcal M_{\SS ,\T}$
 (cf. \eqref{eq:MV}). 
 By construction of $\mathcal R$, the stratum of $\mathcal R$ containing the point $x$ is $\dos SI$ (cf. \eqref{R2}).
Let 
 $\varphi \colon (\R^m \times \rc L, \I \times \rc \L) \to (V,\TT)$  be a $\TT$-conical chart  of $x$ with  link  $(L,\L)$.
 It suffices to prove that $(V,\TT) = (V,\mathcal R)$.
 
 Since $\dos SI$ is an open subset of $\bi SI$ (cf. \lemref{MR} (a)) we can suppose
 \begin{equation}\label{eq:R=T}
\dos SI\cap V =  \bi SI \cap V  = \varphi (\R^m \times \{\tv\}).
\end{equation}
 By  definition of $\mathcal M_{\SS ,\T}$ we have that the only virtual  $\SS$-stratum on $V$ is $ V \cap S$.
 So, there are no virtual $(\SS,\T)$-strata on $V \menos \bi SI$. We conclude from    \lemref{refprop} (a) that
 $(V \menos \bi SI, \SS) = (V \menos \bi SI,\TT)$
 and 
 therefore $(V \menos \bi SI, \SS) = (V \menos \bi SI,\mathcal R)$. Using \eqref{eq:R=T} we get the claim $(V,\TT) = (V,\mathcal R)$.
 \epr

\begin{remarque}{\em
 Notice that the coarsening of a CS-set is not necessarily a CS-set. Let us give an example.
 
 \begin{multicols}{2}
 \includegraphics[scale=0.7]{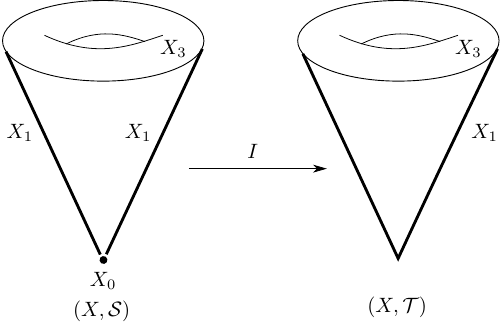}
 
 On the CS-set $(X,\mathcal S)$ the link of the strata of $X_1$ (resp. of the stratum $X_0$) is $S^1$ (resp. $T^2$). This lack of uniformity implies that the coarsening $(X,\mathcal T)$ is not a CS-set.

  \end{multicols}}
 
 \end{remarque}

\rojo
The following Proposition describes the construction of compatible conical charts associated to a simple refinement. We need two preliminary results. First one is a nice result  about cones from Stallings \cite{MR0180983, MR0239611}) (see \cite[Lemma 2.10.1]{LibroGreg} for a proof).

\begin{lemma}\label{lem:conos}
Let $X$ and $Y$ be two compact topological spaces, whose apexes are $\tx$ and $\ty$ respectively. If there is an open neighborhood $U$ of  $\tx$ in  $\rc X$ so
that $(U, \tx) $ and  $(\rc Y,\ty)$ are homeomorphic, then $(\rc X, \tx)$ and $(\rc Y,\ty)$ are homeomorphic.
\end{lemma}

\begin{lemma}\label{lemmaUg}
Let $(X,\SS) \triangleleft_I (X,\mathcal T)$ be a refinement. Consider a stratum $S \in \SS$ and a point $x\in S$. 
Let us suppose $b =\dim \bi SI -\dim S\geq 1$.
Then, there exists 
\begin{enumerate}[i)]
\item an open neighborhood $U\subset S^I$  of $x$, as small as necessary, 
\item two homeomorphisms $ (U,x) \xleftarrow g (\R^a \times \rc \S^{b-1} ,(0,\tv) )  \xrightarrow f (\R^{a+b},0) $ with 
$g^{-1}(S \cap U) = \R^a \times \{\tv\}$.
   \end{enumerate}
 Here, where $a=\dim S$ and $\tv $ is the apex of the cone $\rc \S^{b-1}$.
\end{lemma}
\begin{proof}\rojo
Let  $g\colon (\R^a \times \rc L, \I \times \rc \L) \to (V,\SS)$ be a $\SS$-conical chart  of $x\in S$ where $V$ is small as necessary.
Recall that
 $
S\cap V  = g (\R^a \times \{\tw\})
$
and $g(0,\tw) = x$,
where $\tw$ is the apex of $\rc L$. 
The subset $g^{-1} (V\cap \bi SI \menos S)$ is of the form $\R^a \times A \times ]0,1[$ for some subset $A \subset L$
 (cf. (S6)). We get the homeomorphism $g\colon\R^a \times \rc A \to V \cap \bi SI$. We take $U = V \cap \bi SI$. 
 Notice that $\dim S=a$ and $\dim \bi SI=a+b$.

 Without loss of the generality we can suppose that $U$ is included in an open subset of $\bi SI$  homeomorphic to $\R^{a+b}$. So, the cone $\rc (\S^{a-1} * A)$ is homeomorphic to an open subset of $\rc S^{a+b-1}$.
Then  $\rc(\S^{a-1}* A)$ is homeomorphic to $\rc S^{a+b-1}$ by a homeomorphism preserving the apexes (cf. \lemref{lem:conos}).  This gives $f$. A standard calculation gives that $A$ is a $(b-1)$-homological sphere. 
\end{proof}\negro

\begin{remarque}\label{New}\rm{
Let $x$ be a point of a virtual  stratum  $S$ . Previous Lemma shows that we locally have
$
\left\{
\begin{array}
{lll}
w&=& (0,\tv)\\
S &=& \R^a \times \{ \tv\}\\
\bi SI&=& \R^a \times \rc \S^{b-1}
\end{array}
\right.,
$
where $\S^{b-1}$ is a $(b-1)$-homological sphere and $\tv$ is the apex of the cone. 
This is the difference with the refinement of \cite{SRef1} (cf. \remref{RemS6}). 

The subset $\bi SI\menos S = \R^a  \times \S^{b-1} \times ]0,1[$ is a union of strata $Q\in\SS$.
Notice that $S \prec Q$.

Let us suppose that the refinement $(X,\SS)  \triangleleft_I  (X,\T)$ is simple. Since $\depth \mathcal V=0$ then all the strata of $\bi SI\menos S $ are source strata. If $b>1$ then $\bi SI/S$ is just a source stratum, when $b=1$ the subset $\bi SI\menos S$ is the union of two source strata. 

Globally, $\bi SI$ contains a discret family of strata of $\mathcal V = \mathcal M$, the rest being source strata.}
\end{remarque}

\bp\label{Max}
Let $(X,\SS)  \triangleleft_I  (X,\T)$ be a simple refinement between two CS-sets. 
We consider a point $x \in \Sigma_{\mathcal S}$ and we write $S\in \SS$ and $\bi SI \in \T$ the  strata containing $x$. 
 We distinguish three cases.
\begin{itemize}\setlength\itemsep{.5em}
\item[(a)] \emph{$S$ is a source stratum.} Then there exists 
\begin{itemize}\setlength\itemsep{.3em}
\item[-] a $\SS$-conical chart $(\psi,W)$  of $x\rojo \in S$, whose link is $(L,\mathcal L)$, 
\item[-]  a $\mathcal T$-conical chart $(\psi,W)$  of $x\rojo \in \bi SI$, whose  link is $(L,\L')$ for  \rojo some refinement  $\L'$ of $\L$, and 
\item[-] a
 commutatif diagram
$$
\xymatrix{
(\R^a \times \rc L, \I \times  \rc \mathcal L)  \ar[d]^I \ar[r]^-\psi  & (W,\mathcal S) \ar[d]^I\ \\
(\R^a \times \rc L, \I \times \rc \mathcal L') \ar[r]^-\psi & (W,\mathcal T)\\
}
$$
\end{itemize}
\negro

\item[(b)] \emph{$S$ is an exceptional stratum.} Let $b = \dim \bi SI - \dim S \geq 1$. Then there exists
\begin{itemize}
\item[-]  a  $\SS$-conical chart  $(\rojo \phi \negro, W)$ of  $x \rojo \in S$, whose link is $\rojo ( \S^{b-1},\I)$, 
\item[-] \rojo  a   chart  $(\psi, W)$ of  $x \rojo  \in \bi SI$.
 and
\item[-] a
 commutatif diagram
$$
\xymatrix{
(\R^a \times \rc \S^{b-1}, \I \times \rc \mathcal I)  \ar[d]^f \ar[r]^-\phi  & (W,\mathcal S) \ar[d]^I\ \\
(\R^{a+b}, \I ) \ar[r]^-\psi & (W,\mathcal T),\\
}
$$
where $f$ is a homeomorphism.
\end{itemize}
\negro

\item[(c)]  \emph{$S$ is a virtual stratum and $\bi SI$ is a singular stratum}. Let $b = \dim \bi SI - \dim S \geq 1$. Then there exists 
\begin{itemize}\setlength\itemsep{.3em}
 \item[-] a  $\SS$-conical chart  $(\phi, W)$ of  $x \rojo \in \bi SI$, whose link is $\rojo (\S^{b-1}*E,\mathcal E_{\star b-1})$, 
\item[-] a $\T$-conical chart   $(\psi,W)$ of $x\rojo \in S$, whose link is $(E;\mathcal E)$, and 
\item[-] a
 commutatif diagram
$$
\xymatrix{
(\R^a \times \rc (\S^{b-1} * E), \I \times  \rc \mathcal E_{\star b-1})  \ar[d]_{(f\times \id)(\id\times h^{-1})} \ar[r]^-\phi  & (W,\mathcal S) \ar[d]^I\ \\
(\R^{a+b} \times \rc E, \I \times \rc \mathcal E) \ar[r]^-\psi & (W,\mathcal T)\\
}
$$
where $h,f$ are defined in   \propref{join1} and \lemref{lemmaUg}.
\negro
\end{itemize}
\end{itemize}
\ep
\bpr 
\rojo (a) We consider a $\mathcal S$-conical chart of $x \in \bi SI$:
$$
\psi \colon (\R^a \times \rc L, \I \times \rc \mathcal L) \to (W,\mathcal S)\\
$$
Each stratum of $\mathcal T$ is a union of strata of $\SS$ (cf. (S6)). Since $\dim \bi SI = \dim S$ then $\bi SI \cap W = \psi (\R^a \times \{ \tw\}) = S \cap W$ where $\tw$ is the apex of the cone $\rc L$. The other strata of $(W,\mathcal T)$ are of the form 
$\psi (\R^a \times A \times ]0,1[)$ where $A$ is a union of some strata of $\L$.
So, there exists a filtration $\L'$ on $L$ such that $\psi\colon  (\R^a \times ]0,1[ \times L, \I \times\I \times \L') \to (W\menos \bi SI , \mathcal T)$ is a stratified homeomorphism. This is also the case for $\psi\colon  (\R^a  \times\rc L, \I \times \rc\L') \to (W, \mathcal T)$.
We get that   $(\psi, W) $ is a   $\mathcal T$-conical chart of $x \in \bi SI$ with link $(L,\L')$.
This gives (a).

\medskip

We treat the cases (b) and (c), where $\dim S < \dim \bi SI$.
 We have  $S \in \mathcal V = \mathcal M$, since the decomposition is simple.
 Notice that  $\depth \mathcal M = 0$.  Then we can suppose that $\mathcal M= \{ S \}$, since (b) and (c) are local questions.
 In other words,
$
\SS = \{S \} \sqcup  \mathcal A$. This implies $\SS = \mathcal T$ on $X\menos S$ and therefore on $\bi SI\menos S$.
The stratification $\SS$ induces on $\bi SI$ the stratification
\begin{equation}\label{ST}
 \{ S, (\bi SI\menos S)_{cc}\} \ \ \hbox{ with } S \preceq (\bi SI\menos S)_{cc}
\end{equation}

\smallskip

\rojo 
(b)  Take $(\phi,W) = (g,U)$ and $(\psi,W)= (g\circ f^{-1},U)$ from \lemref{lemmaUg}.
Since  $\bi SI$ is a regular stratum of $\TT$, then $W$ is an open subset of $X$.
The pair $(\psi,W)$ is a chart of $x \in \bi SI$.

Recall that $g^{-1}(U\cap S) = \R^a \times \{ \tv\}$, where $\tv$ is the apex of $\rc \S^{b-1}$.
Then,
$$
\phi \colon (\R^a \times (\S^{b-1} \times ]0,1[) ,\mathcal I \times \mathcal I) \xrightarrow f (f(\R^a \times \S^{b-1} \times ]0,1[) ,\mathcal I) \xrightarrow {g\circ f^{-1}} (U\menos S, \mathcal T) = (U\menos S,\mathcal S)
$$
is a stratified homeomorphism.
From \eqref{ST},  we conclude that 
$$\phi\colon ( \R^{a} \times \rc \S^{b-1}, \I \times \rc \I) \to (W,\mathcal S)$$ is a stratified homeomorphism an therefore $(\phi,W)$ is a $\mathcal S$-chart of $x \in S$ whose link is $(\S^{b-1},\I)$. 

By construction,we have $\phi = \psi \circ f$.
\smallskip

(c) We define $a=\dim S$, which gives $\dim \bi SI =a+b$.
Let   $\varphi \colon  (\R^{a+b} \times \rc E, \I \times \rc \mathcal E) \to (Q,\mathcal T)$   be a $\T$-conical chart of $x \in \bi SI$. The set $Q \cap \bi SI= \varphi (\R^{a+b} \times \{ \tw\})$, where $\tw $ is the apex of the cone $\rc E$, is an open neighborhood of $x \in \bi SI$. We consider $(U,g)$ given by  \lemref{lemmaUg} with $U \subset Q \cap \bi SI$. 
We define $W = \varphi(\pr \varphi^{-1}(U) \times \rc E )\subset Q$
 which is an open subset containing $x$. Here $\pr \colon \R^{a+b} \times \rc E  \to \R^{a+b}$ is the canonical projection.
 The stratified homeomorphism
$$
\gamma= \varphi \circ(  (\pr \circ \varphi^{-1} \circ g )\times \id_{\rc L}) \colon 
( \R^a \times \rc \S^{b-1}  \times \rc E, \I \times  \I \times \rc \mathcal E) \to (W,\mathcal T),
$$
verifies
$\gamma^{-1}(\bi SI) = \R^a \times \rc \S^{b-1}  \times \{ \tw\}$ and 
$\gamma^{-1}(S) = \R^a \times \{ (\tv ,  \tw)\}$.
We consider $f$ given by  \lemref{lemmaUg}.
 Notice that 
$$
\psi = \gamma \circ (f^{-1} \times \id)  \colon 
( \R^{a+b}  \times \rc E,   \I \times \rc \mathcal E) \to (W,\mathcal T)
$$
gives  a $\T$-conical chart   $(\psi,W)$ of $x$, whose link is $(E,\mathcal E)$.

From \eqref{ST} we get  that the stratification $\SS$ induces the stratification $\R^a \times \{( \tv , \tw)\} ,
\R^a \times  (\S^{b-1})_{cc} \times ]0,1[   \times \{ \tv\}$ on $\gamma^{-1}(\bi SI)$. In other words, the map
$$
\gamma  \colon 
( \R^a \times \rc \S^{b-1}   \times \{\tw\} , \I \times \rc \I \times \I) \to (W\cap  \bi SI,\mathcal S)
$$
is a stratified homeomorphism.
Since  all the strata of $(W\menos S,\SS)$ are source strata then $\SS= \mathcal T$ on $W\menos \bi SI$ (cf. \lemref{refprop} (a)).
  This gives that 
$$
\gamma  \colon 
( \R^a \times \rc \S^{b-1}  \times E \times ]0,1[ , \I \times \I \times \mathcal E \times I ) \to (W\menos \bi SI,\mathcal S)
$$
is a stratified homeomorphism.
Combining these two results, we get that
  $$
  \gamma\colon (\R^{a} \times  (\rc \S^{b-1} \times \rc E), \I \times (\I \times \rc \mathcal E)_{(\tv,\tw)}) \to (W,\mathcal S)$$
  is a stratified homeomorphism (cf. \eqref{ptplus}). Finally, using the stratified homeomorphism $h$ given by \propref{join1} we get the conical chart
  $$
\phi = \gamma  \circ(\Id \times h^{-1}) \colon  (\R^a \times \rc (\S^{b-1} * E),  \I \times \rc\mathcal E_{\star b-1} ) 
 \to (W,\mathcal S)
 $$
  whose link is $(\S^{b-1}*E,\mathcal E_{\star b-1})$.
  
  By construction, we have $\psi  (f\times \id)(\id\times h^{-1}) = \gamma  (f^{-1} \times \id) (f\times \id)(\id\times h^{-1})
  =  \gamma (\id\times h^{-1}) = \phi $
\epr

\subsection{\bf Charts and perversities.} \rojo
Consider a CS-set $(X,\mathcal Q)$ and a conical chart 
$$\beta \colon (\R^\ell \times \rc L, \I \times \rc L) \to (V, \mathcal Q)$$
 of a point $x \in Q$, where $Q \in \mathcal Q^{sing}$. A perversity $\ov q$ on $(X,\mathcal Q)$ induces a perversity on the LHS which is described as follows. By restriction, $\ov q$ determines a perversity on $(V,\mathcal Q)$ still denoted by $\ov q$. We call again $\ov q$ the perversity induced on $(\R^\ell \times \rc L, \I \times \rc L)$ by the stratified homeomorphism $\beta$. A such perversity is determined by a perversity on the link $(L,  \L)$, also denoted by $\ov q$, and by the number $\ov q(Q)$ following these formul\ae:
 \begin{equation}\label{eq:conv}
\ov q(\underbrace{\R^\ell \times P \times ]0,1[}_{=_{\beta} V \cap R}) =   \ov q(P)  =  \ov q (R) \ \ \ \hbox{ and } 
\ \ \  \ov q(\R^m \times \{ \tw \}) = \ov q(\tw) = \ov q(Q),   
\end{equation}
where 
where $R \in \mathcal Q$, $R\ne Q$,  and 
$\tw$ is the apex of $\rc L$.\negro

\smallskip

 We study the behavior of the perversities concerning the charts of  \propref{Max}. More precisely, if $I \colon (X,\SS) \to (X,\TT)$ is the stratified map induced by the refinement $(X,\SS)\triangleleft (X,\TT)$ and $\ov p$ is a perversity on $(X,\SS)$ \rojo we study the perversities $\ov p$ and $I_\star \ov p$ \negro under the previous conventions \eqref{eq:conv} following the three cases presented in \propref{Max}.
 
 \smallskip
 
 (a) The map $I \colon (V,\SS) \to (V,\TT)$ becomes \rojo the identity $$I \colon (\R^m \times \rc L, \I \times \L) \to
 (\R^m \times \rc L, \I \times \L'),
 $$\negro which  is a stratified map. Recall that $\varphi(V \cap S) = \R^m \times \{\tv\} = \varphi (V \cap \bi SI)$ 
 (cf. \eqref{eq:R=T}).
  Previous conventions give the equalities 
 \begin{equation} \label{eq:SVIp}
\rojo  \ov p(\R^m\times \{ \tv \}) = \ov p ( \tv)  =  \ov p (S)  \ \hbox{ and } \ \    I_\star \ov p(\R^m\times \{ \tv \})= I_\star\ov p ( \tv) =   I_\star\ov p(\bi SI) .
 \end{equation}

 \smallskip

 (b) The map $I \colon (W,\SS) \to (W,\TT)$ becomes \rojo the stratified homeomorphism 
  $$
  \psi^{-1} \circ \phi = f \colon
 (\R^{a} \times  \rc \S^{b-1} , \I \times  \rc \I) \to 
 (\R^{a+b} , \I )
  $$ \negro
 where the strata $S , (\bi SI\menos S)_{cc}\in \SS$ and $\bi SI \in \T$ become respectively $\R^{a} \times 
\{ \tu \}, \R^{a} \times (\S^{b-1})_{cc} \times ]0,1[$ and $\R^{a+b}$,
where $\tu$ is the apex of the cone $\rc \S^{b-1}$. \negro
We have $I_\star \ov p= \rojo \ov 0$ and previous convention \eqref{eq:conv} gives the formula
\begin{equation}\label{eq:b}
\begin{array}{ccccc}
\rojo \ov p (\R^a\times \{\tu\})&= &\ov p (\tu) &= & \ov p(S) \\[,2cm] 
 \ov p (\R^a\times (\S^{b-1})_{cc} \times ]0,1[)&=& \ov p ((\S^{b-1})_{cc} ) &=&  \ov p((\bi SI\menos S)_{cc}) ,
 \end{array}
 \end{equation}
\rojo where $\tu$ is the apex of the cone $\rc \S^{b-1}$. \negro
 \smallskip

 (c) The map $I \colon (W,\SS) \to (W,\TT)$ becomes the stratified homeomorphism
 \bee\label{eq}
 \psi^{-1} \circ \phi  
\colon (\R^{a} \times  \rc (\rojo \S^{b-1}\negro * E), \I \times  \rc \mathcal E_{\star b-1}) \to 
 (\rojo \R^{a+b} \times  \rc  E, \I \times \rc \mathcal E)
 \eee
 given by \rojo
$(f\times \id ) \circ (\id \times h)$ (cf. \eqref{refirefi} and \lemref{lemmaUg} ). \negro
 The strata $S , (\bi SI\menos S)_{cc}\in \SS$ and $\bi SI \in \T$ become respectively $\R^{a} \times 
\{ \tu \}, \R^{a} \times  \left(S^{b-1}\right)_{cc} \times ]0,1[$ and $\rojo \R^{a+b}  \times \{ \tw\}$,
where $\tu$ is the apex of the cone $\rc (\rojo \S^{b-1} \negro * E)$ and $\tw$ is the apex of the cone $\rc  E$.
The other strata oh the LHS are source strata.
Previous convention \eqref{eq:conv} gives the formul\ae

\begin{equation} \label{eq:SRQT}
\begin{array}{ccccc}
\ov p(\underbrace{\R^{a} \times  \rojo \rc \S^{b-1}\negro \times Q \times ]0,1[}_{=_\phi \rojo  W \negro\cap T}) &=&\ov p (Q)
 &= & \ov p(T) \\[,2cm]
  I_\star \ov p (\underbrace{\rojo \R^{a+b} \negro \times  Q \times ]0,1[}_{=_{\psi} \rojo  W \negro \cap \bi TI})&=& I_\star \ov p (Q) &=&  I_\star \ov p (\bi TI)  \\[,2cm] 
\ov p ( \R^{a} \times  \left(\rojo \S^{b-1} \negro \right)_{cc} \times ]0,1[) &=&  \ov p(  \left(\rojo \S^{b-1} \negro \right)_{cc} ) &=&  \ov p ((\bi SI\menos S)_{cc})  \\[,2cm]
 \ov p (\R^a\times \{\tu\})&=& \ov p (\tu) &=&  \ov p(S)  \\[,2cm]
 I_\star \ov p ( \rojo \R^{a+b} \negro \times \{ \tw \}) 
 & =& I_\star \ov p(\tw)  &=& I_\star \ov p (\bi SI) 
\end{array}
\end{equation}
where $\rojo T \in \mathcal T$ with $\bi TI \ne \bi SI$.

\subsection{Comparison tools} The invariance results we prove in two final sections follow the same pattern: a stratified map induces an isomorphism in (co)homology. To achieve this objective we use these two results. The first one is used with compact supports (cf.  \cite[Theorem 5.1]{CST3}) and the second one is used with closed supports (cf. \cite[Proposition 13.2]{CST5}).

\bp\label{Met1}
Let $\mathcal F_{X}$ be the category whose objects are (stratified homeomorphic to) open subsets
of a given  CS set $(X,\SS)$ and whose morphisms are  stratified homeomorphisms and inclusions.
Let  $\mathcal Ab_{*}$ be the category of graded abelian groups. Let $F_{*},\,G_{*}\colon \mathcal F_{X}\to \mathcal Ab$
be two functors and
 $\Phi\colon F_{*}\to G_{*}$ a natural transformation satisfying
 the conditions listed
below.
\begin{enumerate}[\rm (a)]

\item $F_*$ and $G_{*}$ admit exact Mayer-Vietoris sequences and the natural transformation $\Phi$ 
 induces a commutative diagram between these sequences,

\item If $\{U_{\alpha}\}$ is a increasing collection of open subsets of $X$  and $\Phi\colon F_{*}(U_{\alpha})\to G_{*}(U_{\alpha})$ is an isomorphism for each $\alpha$, then $\Phi\colon F^{*}(\cup_{\alpha}U_{\alpha})\to G^{*}(\cup_{\alpha}U_{\alpha})$  is an isomorphism.

\item Consider $(\varphi,V)$ a conical chart of a singular point $x \in S$ with $S \in \SS$.
If 
$\Phi\colon F^{*}(V\menos S)\to G^{*}(V\menos S)$
is an isomorphism, then so is
$\Phi\colon F^{*}(V)\to G^{*}(V)$. 
%
%
%\item If $L$ is a compact filtered space such that 
%$X$  has an open subset  stratified homeomorphic
%to $\R^i\times \rc L$ and, if
%$\Phi\colon F^{*}(\R^i\times (\rc L\backslash \{\tv\}))\to G^{*}(\R^i\times (\rc L\backslash \{\tv\}))$
%is an isomorphism, then so is
%$\Phi\colon F^{*}(\R^i\times \rc L)\to G^{*}(\R^i\times \rc L)$. Here, $\tv$ is the apex of the cone $\rc L$.

\item If $U$ is an open subset of X contained within a single stratum and homeomorphic
to an Euclidean space, then $\Phi\colon F^{*}(U)\to G^{*}(U)$ is an isomorphism.
\end{enumerate}
Then $\Phi\colon F^{*}(X)\to G^{*}(X)$ is an isomorphism.
\ep

\begin{proposition}\label{Met2}
Let $\mathcal F_{X}$ be the category whose objects are (stratified homeomorphic to) open subsets
of a given paracompact second countable\footnote{In the original reference \cite[Proposition 13.2]{CST5} the pseudomanifold $X$ needs to be  separable. 
 A second countable space is separable (see for example \cite[Theorem 16.9]{Wil}) so we can change this last hypothesis in the statement of the Proposition.}
CS-set $X$ and whose morphisms are  stratified homeomorphisms and inclusions.
Let  $\mathcal Ab_{*}$ be the category of graded abelian groups. Let $F^{*},\,G^{*}\colon \mathcal F_{X}\to \mathcal Ab$
be two functors and
 $\Phi\colon F^{*}\to G^{*}$ a natural transformation satisfying
 the conditions listed
below.
\begin{enumerate}[(a)]

\item $F^{*}$ and $G^{*}$ admit exact Mayer-Vietoris sequences and the natural transformation $\Phi$ 
 induces a commutative diagram between these sequences,

\item If $\{U_{\alpha}\}$ is a disjoint collection of open subsets of $X$  and $\Phi\colon F_{*}(U_{\alpha})\to G_{*}(U_{\alpha})$ is an isomorphism for each $\alpha$, then $\Phi\colon F^{*}(\bigsqcup_{\alpha}U_{\alpha})\to G^{*}(\bigsqcup_{\alpha}U_{\alpha})$  is an isomorphism.

\item Consider $(\varphi,V)$ a conical chart of a singular point $x \in S$ with $S \in \SS$.
If 
$\Phi\colon F^{*}(V\menos S)\to G^{*}(V,\menos S)$
is an isomorphism, then so is
$\Phi\colon F^{*}(V)\to G^{*}(V)$. 
%
%\item Consider a conical chart If $L$ is a compact filtered space such that 
%$X$  has an open subset  stratified homeomorphic
%to $\R^i\times \rc L$ and, if
%$\Phi\colon F^{*}(\R^i\times (\rc L\backslash \{\tv\}))\to G^{*}(\R^i\times (\rc L\backslash \{\tv\}))$
%is an isomorphism, then so is
%$\Phi\colon F^{*}(\R^i\times \rc L)\to G^{*}(\R^i\times \rc L)$. Here, $\tv$ is the apex of the cone $\rc L$.
%

%
%\item If $L$ is a compact filtered space such that 
%$X$  has an open subset  stratified homeomorphic
%to $\R^i\times \rc L$ and, if
%$\Phi\colon F^{*}(\R^i\times (\rc L\backslash \{\tv\}))\to G^{*}(\R^i\times (\rc L\backslash \{\tv\}))$
%is an isomorphism, then so is
%$\Phi\colon F^{*}(\R^i\times \rc L)\to G^{*}(\R^i\times \rc L)$. Here, $\tv$ is the apex of the cone $\rc L$.

\item If $U$ is an open subset of X contained within a single stratum and homeomorphic
to an Euclidean space, then $\Phi\colon F^{*}(U)\to G^{*}(U)$ is an isomorphism.
\end{enumerate}
Then $\Phi\colon F^{*}(X)\to G^{*}(X)$ is an isomorphism.
\end{proposition}

\begin{remarque}\label{Uno} {\em
A priori, in order to apply \propref{Met1} and \propref{Met2}   one needs to verify condition (c)  for any conical chart of $X$. Reading carefully the proof of these Propositions one notices that it is enough to verify (c)  for a neighborhood basis of each point $x$ of $X$. 

Associated to a conical chart $\varphi \colon \R^i \times \rc L \to V$ of the point $x$, we can  construct a neighborhood basis
$B_x =\big\{ \varphi_\varepsilon \colon ]-\varepsilon, \varepsilon [^i \times \rc_\varepsilon L \to V_\varepsilon \mid \varepsilon >0 \big\}$
of $x$, where $\rc _\varepsilon = L \times [0,\varepsilon[ / L \times \{ 0\}$ and $V_\varepsilon = \varphi(]-\varepsilon, \varepsilon [^i \times \rc_\varepsilon L )$. Notice that all this open subsets are stratified homeomorphic after homotethy.

So, in order to apply \propref{Met1} and \propref{Met2}   it suffices to verify conditions (c) for a conical chart of each point of $X$.

Notice that the family $F_x =\big\{ \varphi_\varepsilon \colon [-\varepsilon, \varepsilon ] \times \tc_\varepsilon L \to V_\varepsilon \mid \varepsilon >0 \big\}$ is a 1neighborhood basis for the point $x$ made up of closed subsets. In other words, the space $X$ is locally compact.

}\end{remarque}

\subsection{ Morphisms.}\label{subsec:morfismo}
Consider a  stratified map $f \colon (X,\SS,\ov p) \to (Y,\T,\ov q)$ between two perverse CS-sets. If the perversities verify $f^* D\ov q \leq D \ov p$ then we have the following induced morphisms.

\begin{enumerate}[(a)]

\item 
$ f_* \colon \lau H {\ov p} * {X;\SS} \to \lau H {\ov q} * {Y;\T}
$ and
$ f^* \colon \lau H *{\ov p}  {X;\SS} \to \lau  H * {\ov q}  {Y;\T}
$
(cf. \cite[Proposition 3.11]{CST3}).
 
 \item 
$ f^* \colon \lau H *{\ov p,c} {X;\SS} \to \lau H *{\ov q,c}  {Y;\T} 
$
if the map $f$ is a proper map. This comes from \ref{mpti}.d and from the fact that the family $\{f^{-1}(K) \mid K \subset Y \hbox{ compact}\}$ is cofinal in the family of compact subsets of $X$ if $f$ is proper.

\item 
$ f_* \colon \lau \gH {\ov p} * {X;\SS} \to \lau \gH {\ov q} * {Y;\T},
$
 if $f(\dos X {\ov{p}}) \subset \Sigma_{(Y,\T)}$ 
where $\dos X {\ov p} =\sqcup \{\ov S \mid S \in \SSsing \hbox{ and } \ov p (S) > \ov t (S)\}$ (cf. \cite[Proposition 3.11]{CST3}).
An adapted version of this result is needed in this work (see \lemref{masfuerte}).

\item 
$ f_* \colon \lau H {BM,\ov p}* {X;\SS}  \to  \lau H {BM,\ov q}* {X,\TT} 
$
and
$ f_* \colon \lau \gH {BM,\ov p}* {X;\SS}  \to  \lau \gH {BM,\ov q}* {X,\TT} 
$
(cf. \ref{mpti}.d). \negro

\end{enumerate}
If the perversities verify $f^* \ov q \leq  \ov p$ then we have the following induced morphisms:

\begin{enumerate}[(e)]

\item 
$ f^* \colon \lau \IH * {\ov q}  {Y;\T} \to \lau \IH * {\ov p}  {X;\SS}
$
if 
$f^\star\ov q \leq  \ov p$
(cf. \cite[Theorem A]{CST5}).

\item [(f)]
$ f^* \colon \lau \IH * {\ov q,c}  {Y;\T} \to \lau \IH * {\ov p,c}  {X;\SS}
$
if the map $f$ is a proper map.This comes from \ref{mpbi}.d and from the fact that the family $\{f^{-1}(K) \mid K \subset Y \hbox{ compact}\}$ is cofinal in the family of compact subsets of $X$ if $f$ is proper.

\end{enumerate}
\negro

\section{Refinement invariance for CS-sets}
\begin{quote}
We prove the main result of this work: the refinement invariance of all the homologies and cohomologies of \secref{homcoho}  :  \thmref{thm:Coars} for coarsenings and \thmref{thm:refi} for refinements. In the first case, we need to work with a particular type of perversities, the $K$-perversities. \end{quote}

\subsection{$K$-perversities} These are the perversities for which refinement invariance holds. Roughly speaking, they are $M$-perversities defined on the LHS of a refinement $(X,\SS) \triangleleft (X,\TT)$ whose restriction to the strata of the RHS is a classical perversity verifying the growing condition of a Goresky-MacPherson perversity.

\begin{definition}
Let $(X,\SS) \triangleleft (X,\mathcal T)$ be a refinement. A perversity $\ov p$ on $(X,\SS)$ is a \emph{$K$-perversity}  if it verifies conditions (K1) and (K2).
\begin{enumerate}[(K1)]
\item We have, for any strata $S,Q \in \SS$ with $S\preceq Q$ and $\bi SI = \bi QI$,
\begin{equation}\label{eq:magica}
 \ov p(Q) \leq \ov p(S) \leq \ov p(Q) + \ov t(S) - \ov t(Q),
\end{equation}
\item We have, for any strata $S,Q \in \SS$ with $\dim S =\dim Q$ and $\bi SI = \bi QI$,
\begin{equation}\label{eq:magicaTris}
 \ov p(Q) = \ov p(S) ,
\end{equation}
\end{enumerate}
\end{definition}
\begin{remarque}\label{rem:Kper}
{\rm
Notice these two conditions are equivalent to conditions
\begin{equation}\label{eq:magicaBis}
 D\ov p(Q) \leq D\ov p(S) \leq D\ov p(Q) + \ov t(S) - \ov t(Q) \ \ \hbox{ and } D \ov p(Q) = D \ov p(S).
\end{equation}
Also, condition \eqref{eq:magica}  is always verified when both strata $S$ and $Q$ are regular strata. If the stratum $Q$ is regular and the stratum $S$ is singular (id est, $S$ is an exceptional stratum), then  condition \eqref{eq:magica} becomes
\begin{equation}\label{eq:regS}
0 \leq \ov p(S) \leq \ov t(S).
\end{equation} In particular, the existence of a $K$-perversity implies the non-existence of 1-exceptional strata  since $0\leq \ov t(S)=-1$ is not possible.
}\end{remarque}

\medskip

Before proving the main results of this work, we need some technical Lemmas.

\begin{lemma}\label{lem:I*p}
Let  $(X,\SS)  \triangleleft_I  (X,\T)$ be a refinement. Any $K$-perversity $\ov p$ verifies
$ 
I_\star\ov p (T) =  \ov p (S)
$ 
for each $T \in \TT$ where $S \in \SS$ is a source stratum of $T$.
\end{lemma}
\begin{proof}
We know from  \secref{subsec:perv} that $I_\star\ov p (T) = \min \{ \ov p (Q) \mid Q \in \SS \hbox{ and } \bi QI = T \}$.
For any $Q \in \SS$ with $\bi QI =T$ there exists a source stratum $S \in \SS$  with $Q\preceq S$ and $\bi SI = T$ (cf. \lemref{refprop} (b)). 
So, $I_\star\ov p (T) =_{(K1)}  \min \{ \ov p (S) \mid S \in \SS \hbox{ source stratum of } T \}$. Condition
\eqref{eq:magicaTris}  ends the proof.
\end{proof}

\bl\label{ll}
Let  $(X,\SS) \triangleleft  (X,\T)$ be a refinement.
For any  $K$-perversity  $\ov p$ we have
  $I^\star  DI_\star\ov p \leq D\ov p$.
 \el
\bpr
Given a stratum $ S \in \SS$, there exists  a source stratum $Q \in \SS$ of $\bi SI$ verifying $S\preceq Q$ (cf. \lemref{refprop} (b)).
We have
\begin{eqnarray*}
I^\star D I_\star\ov p(S) &= &DI_\star\ov p (\bi SI) =
\ov t(\bi SI) - I_\star\ov p(\bi SI) \stackrel{source} =
\ov t(\bi QI) - \ov p(Q) \leq_{(1)} \ov t( Q) - \ov p(Q) \\
&=& D\ov p(Q) \stackrel{(K1)} \leq D\ov p (S),
\end{eqnarray*}
where $_{(1)}$ comes from \eqref{eq:desing} except when  $Q$ is an exceptional stratum.
In this case $\codim Q \geq 2$ and therefore $\ov t (\bi QI )=0 \leq \ov t (Q)$ (cf. \remref{rem:Kper}).
\epr

\bl\label{masfuerte}
Let  $(X,\SS)  \triangleleft_I  (X,\T)$ be a refinement. between two CS-sets. For any $K$-perversity $\ov p$  we have the induced morphisms $I_\star \colon \lau \gH  {\ov  p} * {X;\SS} \to \lau \gH  {I_\star\ov p}* {X;\mathcal T}$, 
$I^\star \colon  \lau \gH  * {I_\star\ov p}  {X;\mathcal T} \to \lau \gH  * {\ov  p}    {X;\SS} $, $I^\star \colon  \lau \gH  * {I_\star\ov p,c}  {X;\mathcal T} \to \lau \gH  * {\ov  p,c}    {X;\SS} $.
\el
\bpr

If we prove that  the operator $I_\star \colon \lau \gC  {\ov  p} * {X;\SS} \to \lau \gC  {I_\star\ov p}* {X;\mathcal T}$ 
is well defined then, by duality, the operator
$I^\star \colon  \lau \gC  * {I_\star\ov p}  {X;\mathcal T} \to \lau \gC  * {\ov  p}    {X;\SS} $ is also well defined.
Following \cite[Proposition 3.11]{CST3} and  \lemref{ll} it suffices to prove $I(\dos X {\ov p}) \subset \Sigma_{(X,\T)}$. If this is not true, then there exist $Q\in \SS$ and $S \in \SSsing$ with $Q \preceq S$, $\ov p(S) > \ov t (S)$ and $\bi QI \in\TTreg$. Since $\bi QI \preceq_{\eqref{SQSIQI}} \bi SI$ then $\bi SI \in \SSreg$. Then $S$ is an exceptional stratum. This is impossible (cf.  \eqref{eq:regS}).
Last point  comes  from \ref{mpti}.d.
\epr

\bl\label{DI} 
Let $(X,\SS)  \triangleleft_I(X,\mathcal R)  \triangleleft_J(X, \T)$  be two  refinements.
If  $\ov p$ is a $K$-perversity on $(X,\SS)$, relatively to the refinement $E = J \circ I$,  then 
\begin{itemize}
\item[(a)] $\ov p$ is a $K$-perversity, relatively to the refinement $I$, and 
\item[(b)] $I_\star\ov p$ is a $K$-perversity, relatively to the refinement $J$.
\end{itemize}
\el
\bpr
Property (a) comes directly from the fact that $\bi SI = \bi SJ$  implies $
\bi SE = \bi{\bi SI} J = \bi{\bi QI} J = \bi QE$,  if $S,Q \in \SS$. Let us prove (b) in two steps.\smallskip

(K1)$_{I_\star\ov p}$ Consider  $S , Q \in \mathcal R$ with   
 $S \preceq  Q $ and $\bi SJ = \bi QJ$.  \lemref{refprop} (c)  gives two $I$-source strata $S',Q' \in \SS$, of $S$ and $Q$ respectively, with  $S' \preceq Q'$. We have 
 $I_\star\ov p(S) = \ov p(S')$ and $I_\star\ov p(Q) = \ov p(Q')$ (cf. \lemref{lem:I*p}) and 
 \begin{equation}\label{eq:ayuda}
 \bi {S'}E = 
\bi{ \bi {S'}I }J  = \bi SJ = \bi QJ = \bi{ \bi {Q'}I }J  = \bi {Q'}E
\end{equation}
 then  
\begin{eqnarray*}
I_\star\ov p(Q) &=& \ov p(Q')  \stackrel{(K1)_{\ov p}}  \leq \ov p(S') = I_\star \ov p(S) \stackrel{ (K1)_{\ov p}}  \leq \ov p(Q') + \ov t( S') - \ov t (Q') 
\\
&\stackrel{source \, strata} =&
 I_\star\ov p (Q) + \ov t( \bi {S'}I )-\ov t(\bi {Q'}I )
=
 I_\star\ov p (Q) + \ov t (S) -\ov t(Q).
\end{eqnarray*}

\smallskip

(K2)$_{I_\star\ov p}$ Consider  two strata $S , Q \in \mathcal R$  with $\dim S = \dim Q$ and  $\bi SJ = \bi QJ$. 
  \lemref{refprop} (b) gives two source strata $S',Q' \in \SS$ of $S$ and $Q$ respectively, relatively to the refinement $I$. Then $\bi {S'}E   = \bi {Q'}E$ (cf. \eqref{eq:ayuda}) and
  $
\dim S' = \dim \bi {S'}I  =  \dim S =\dim Q= \dim \bi {Q'}I   = \dim Q'.
$
Applying  (K2)$_{\ov p}$ we get  $\ov p(S') = \ov p(Q')$. On the other hand,  \lemref{lem:I*p} gives 
 $I_\star\ov p(S) = \ov p(S')$, $I_\star\ov p(Q) = \ov p(Q')$ and therefore we get the claim
$I_\star\ov p(S) = I_\star\ov p(Q)$.
\epr

\subsection{Main results} We give the two invariance results of the various intersection (co)homologies: by coarsening and by refinement.

\bt[\bf Invariance by coarsening]\label{thm:Coars}
Let $(X,\SS) \triangleleft (X,\mathcal T)$ be a refinement between two CS-sets. For any $K$-perversity $\ov p$ on $(X,\SS)$ the identity
$I \colon X \to X$ induces the isomorphisms
\medskip

\begin{tabular}{lcclc}
\hbox{\rm (R1)}  &  $\lau H {\ov  p}* {X;\SS} \cong \lau H {I_\star\ov p} * {X;\T}$,
 & \color{white}............&
\hbox{\rm (R2)} & $\lau H  * {\ov  p} {X;\SS} \cong \lau H * {I_\star\ov p}  {X;\T}$,
\\[,2cm]
\hbox{\rm (R3)}  &  $\lau H * {\ov  p,c} {X;\SS} \cong \lau H * {I_\star\ov p,c}  {X;\T}$,
 & \color{white}............& 
\hbox{\rm (R4)}  &  $\lau \gH {\ov  p} *{X;\SS} \cong \lau \gH { I_\star\ov p} * {X;\T}$, 
\\[,2cm]
\hbox{\rm (R5)} & $ \lau  \gH * {\ov  p} {X;\SS} \cong \lau \gH * {I_\star\ov p}  {X;\T} $,
 & \color{white}............&
\hbox{\rm (R6)}  &  $\lau \gH * {\ov  p,c} {X;\SS} \cong \lau \gH * {I_\star\ov p,c}  {X;\T}$, 
\end{tabular}

\medskip

\noindent If in addition, $X$ is  second countable then 

\medskip

\begin{tabular}{lcclc}
\hbox{\rm (R7)}  &  $\lau H {BM,\ov  p}* {X;\SS} \cong \lau H {BM,  I_\star\ov p} * {X;\T}$, 
 & \color{white}......&
\hbox{\rm (R8)} & $\lau \gH {BM,\ov  p}* {X;\SS} \cong \lau \gH {BM,  I_\star\ov p} * {X;\T}$,
\\[,2cm]
\hbox{\rm (R9)}  &  $ \lau \IH * {I_\star \ov p} {X;\T} \cong \lau \IH*  {\ov p}  {X;\SS}$, 
& \color{white}......&
\hbox{\rm (R10)}&$ \lau  \IH * {\ov  p,c} {X;\SS} \cong \lau \H * {I_\star\ov p,c}  {X;\T}$
\end{tabular}
\et
\bpr
Notice first that the identity $I$ induces the morphisms (R1), \ldots, (R10). This comes from \lemref{ll}, \ref{subsec:perv}, Paragraph \ref{subsec:morfismo} and \lemref{masfuerte}.
We proceed in several steps.

\medskip

\emph{ (R2) and (R5)}. Apply  the Universal coefficient Theorem of \ref{mpti}.e to (R1) and (R4).

\medskip

\emph{ (R3) and (R6)}. Considering \eqref{iccs}  it suffices to 
prove that $I$ induces  the isomorphisms
  $\lau  H* {\ov  p} {X,X\menos K;\SS} \cong \lau H * {I_\star \ov p}  {X, X\menos K;\T}$ and   $\lau \gH *{\ov  p} {X,X\menos K;\SS} \cong \lau \gH { I_\star \ov p} * {X,X\menos K;\T}$, for each compact subset $K \subset X$. Properties (R2), (R5) and the long exact sequences  of \ref{mpti}.d give the result.

\medskip

\emph{ (R7) and (R8)}. 
Since $X$ second countable  then it is hemicompact (see \cite[Remark 1.3]{ST1}).
  Considering \eqref{limpro} it suffices to 
prove
  $\lau H {\ov  p}* {X,X\menos K\SS} \cong \lau H { I_\star \ov p} * {X, X\menos K;\T}$ and 
 $\lau \gH {\ov  p} *{X,X\menos K;\SS} \cong \lau \gH {I_\star \ov p} * {X,X\menos K;\T}$,
where $K$ is a compact subset of $X$. Properties (R1), (R4) and the long exact sequences  of \ref{mpti}.d give the result.

\medskip

\emph{ (R10)}.
Since $X$ second countable  then it is hemicompact, paracompact and therefore normal (see \cite[Remark 1.3]{ST1}, \cite[Theorem 20.10]{Wil}).
Considering   \propref{limiteiny}  it suffices to 
prove
   $\lau \IH * {\ov  p} {X,X\menos K;\SS} \cong \lau \IH * {I_\star \ov p}  {X, X\menos K;\T}$,
where $K \subset X$ is compact. Property (R9) and the long exact sequence  of \ref{mpbi}.d give the result.

 \medskip
 
\emph{ (R1), (R4) and (R9).} Without loss of generality we can suppose that the refinement is simple (cf. \propref{12} and \lemref{DI}).  
We verify the conditions of \propref{Met1},  for (R1) and (R4), and \propref{Met2}, for (R9).
Th functor $\Phi$ comes from $I\colon X \to X$.

\smallskip

\smallskip

(a) It suffices to consider the Mayer-Vietoris sequences of \ref{mpti}.a,  \ref{mpbmti}.a and \ref{mpbi}.a\footnote{Notice that $X$ is second countable, Hausdorff and locally compact (\remref{Uno}). Then, the pseudomanifold $X$ is  paracompact (cf. \cite[II.12.12]{Bre}).}.

\smallskip

(b) The chains  have compact support, so we get (R1) and (R4). The case (R9) is immediate. 

\smallskip

(d) Since  $\SS_U = \I$ implies $\T_U = \I$ then property (D) becomes a tautology.

\smallskip

(c) Consider a singular point $x \in X$. Following \remref{Uno} we distinguish three cases.

\medskip

\hspace{0cm} (C-a) {\boldmath $x \in S$}, {\bf source stratum of} {\boldmath $\SS$}.  Considering \propref{Max} (a) 
and using the local calculations \ref{mpti}.b and \ref{mpbi}.b, we need to prove
$$
\begin{array}{llll}
(R1) &  \lau H {\ov  p}* {L,\L}  {\cong}  \lau H {I_\star \ov p}* {L,\L'}  
 &\Longrightarrow&
 \lau H {\ov  p}* {\rc L,\rc\L}  {\cong}  \lau H {I_\star \ov p}* {\rc L,\rc\L'}
 \\[,3cm]
(R4) &  \lau \gH {\ov  p}* {L,\L}  {\cong}  \lau \gH {I_\star \ov p}* {L,\L'}  
 &\Longrightarrow&
 \lau \gH {\ov  p}* {\rc L,\rc\L}  {\cong}  \lau \gH {I_\star \ov p}* {\rc L,\rc\L'}
\\[,3cm]
(R9) & \lau \IH {\ov  p}* {L,\L}  {\cong}  \lau \IH {I_\star \ov p}* {L,\L'}  
 &\Longrightarrow&
 \lau \IH {\ov  p}* {\rc L,\rc\L}  {\cong}  \lau \IH {I_\star \ov p}* {\rc L,\rc\L'}.
 \end{array}
 $$
  Since the perversity $\ov p$ verifies $\ov p (S) = I_\star \ov p (\bi SI)$ (cf. \lemref{lem:I*p}) then we have 
  $\ov p(\tv) =I_\star \ov p(\tv)$ (cf. \eqref{eq:SVIp}). The result comes now directly from  the local calculations \ref{mpti}.b and \ref{mpbi}.b.

\medskip

\hspace{0cm}(C-b) {\boldmath$x \in S$, {\bf exceptional stratum of} {\boldmath $\SS$}}.  
Considering   \propref{Max} (b)  and using the local calculations \ref{mpti}.b and \ref{mpbi}.b, we need to prove 

\medskip

$
 (R1) \  \lau H{\ov  p} * {\rc S^{b-1},   \rc \I } {\cong} G, \hfill
  (R4) \ \lau \gH{\ov  p} * { \rc S^{b-1},   \rc \I }  {\cong} G, \hfill
    (R9) \     \lau \IH *{\ov  p}  { \rc S^{b-1},   \rc \I }  {\cong} R.
$
\medskip

 \noindent where 
 $b=\codim S \geq 1$. 
  Since  $0 \leq \ov p (S) \leq \ov t(S) =b-2$ (cf. \eqref{eq:regS}) then we have 
  $0 \leq \ov p(\tu) \leq b-2$ (cf. \eqref{eq:b}). The result comes now directly from  the local calculations \ref{mpti}.b and \ref{mpbi}.b.

\medskip

\hspace{0cm}(C-c) 
 {\boldmath$x \in S$, {\bf virtual stratum, with $\bi SI$ singular stratum of} {\boldmath $\SS$}}.  Considering \propref{Max} (c) and using the  local calculations \ref{mpti}.b and \ref{mpbi}.b,  we need to prove
$$
\begin{array}{llcl}
(R1) &  \lau H {  \ov p}* { \rc (\S^{b-1} *E),  \rc \EE_{\star b-1}}  & {\cong} &  
 \lau H { I_\star \ov p}{*} {\rc E,  \rc\EE}  
 \\[,3cm]
(R4) &  \lau \gH {\ov p}* {\rc (\S^{b-1} *E),  \rc \EE_{\star b-1}}  &{\cong} &  
 \lau \gH {I_\star \ov p}* { \rc E, \rc\EE}
\\[,3cm]
(R9) & \lau \IH *{\ov p}{ \rc (\S^{b-1} *E), \rc \EE_{\star b-1}}  &{\cong} &  
\lau \IH *{I_\star \ov p} { \rc E, \rc\EE},
 \end{array} 
 $$
 where
  $b= \dim \bi SI  - \dim S \geq 1$.
  
   Since $S \preceq (\bi SI\menos S)_{cc}$ (cf. \eqref{ST}) and $\ov p ((\bi SI\menos S)_{cc}) \leq   \ov p (S) \leq \ov p((\bi SI\menos S)_{cc})  + b $ (cf. \eqref{eq:magica}) then we have 
  $\ov p ((\S^{b-1})_{cc}) \leq   \ov p (\tu) \leq \ov p ((\S^{b-1})_{cc})  + b $
   (cf. \eqref{eq:SRQT}). Similarly, we get 
   $D\ov p ((\S^{b-1})_{cc}) \leq   D\ov p (\tu) \leq D\ov p ((\S^{b-1})_{cc})  + b $
   (cf. \eqref{eq:magicaBis}).
   
 \rojo On the other hand, we have $I_\star \ov p(\tw) = I_\star \ov p (\bi SI) =  \ov p ((\bi SI\menos S)_{cc}) =  \ov p(  \left(S^{b-1}\right)_{cc} ) $ (cf. \lemref{lem:I*p} and  \eqref{eq:SRQT})). Since $\dim (\R^a \times \S^{b-1} \times ]0,1[) = \dim( \R^{a+b} \times \{\tw\})$ then   $D I_\star \ov p(\tw) = D\ov p(  \left(S^{b-1}\right)_{cc} )$. 
   We conclude that 
    $$D I_\star \ov p(\tw)   \leq   D\ov p (\tu) \leq D I_\star \ov p(\tw)  + b .
    $$
     \negro
  
   Applying   the local calculations \ref{mpti}.b,c and \ref{mpbi}.b,c the question becomes
   
   \medskip
   
   $
   (R1) \  \lau H {  \ov p}* { E,   \EE }   {\cong}   
 \lau H { I_\star \ov p}{*} {E,  \EE}  , \hfill
(R4) \  \lau \gH {\ov p}* {E,  \EE}    {\cong}  
 \lau \gH {I_\star \ov p}* {E,  \EE}  \hfill
(R9) \  \lau \IH *{\ov p}{E,  \EE}    {\cong}   
\lau \IH *{I_\star \ov p} {E,  \EE}  ,
$

     \medskip
     
\noindent The stratum $S$ belongs to $\mathcal V = \mathcal M$ (cf. \eqref{eq:MV}). Since any other $R \in\SS$ meeting the conical chart $W$ verifies $S \prec R$ then $R$ is a source stratum and then $\ov p(R) = I_\star \ov p (R)$ (cf. \lemref{lem:I*p}). From \eqref{eq:SRQT} we get $\ov p = I_\star \ov p$ on $E$.
The claim is proved.
\epr

\begin{remarque} \label{rem:cod1excep}
{\rm 
The existence of 1-exceptional strata may impeach the above isomorphisms. This is the case for (R4), \dots (R10). For example
$
\lau \gH  {\ov 0} *{\rc S^0,\rc \I} = 0 \ne G = \lau \gH {\ov 0} * {]-1,1[,\I}.
$
But we have 
$
\lau H {\ov 0} * {\rc S^0} =  G = \lau H{\ov 0} *{]-1,1[,\I,\rc \I}.
$
In fact, the local calculations $\lau H {\ov p} 0{\rc S^{0},\SS}$ and $\lau \gH  {\ov p} 0 {\rc S^{0};\SS}$ are different:
$$
\lau H {\ov p} 0 { \rc S^{0},  \rc \I} =
\left\{
\begin{array}{cl}
\hiru H 0 {S^{0}} &\text{ if } D \ov{p}( \tv)  \geq 0 \\
G&\text{ if }  D \ov{p}( \tv) < 0
\end{array}\right.
\phantom{--}
\lau  \gH  {\ov{p}} 0 {  \rc S^{0}, \rc\I }=
\left\{
%:
\begin{array}{cl}
\hiru H 0 {S^{0}}&\text{ if } D\ov{p}( \tv) \geq 0 
\\
0&\text{ if }  D \ov{p}( \tv) < 0.
\end{array}\right.
$$

 We observe that  condition $\lau \gH{\ov  p} * {\rc S^{0},   \rc \I } {\cong} G$ of (C-c) is never fulfilled, while   we just need $D \ov p(\tv) < 0$ to have $\lau H{\ov  p} * {\rc S^{0},   \rc \I } {\cong} G$ of (C-c).

 Condition \eqref{eq:magica} can be weakened
in cases (R1), (R2) and (R3) as follows: dealing with 1-exceptional strata $S$, it suffices to ask $D\ov p(S) < 0$, that is, $\ov p(S) \geq 0$ and not $0 \leq \ov p(S) \leq \ov t(S)$. So, these strata are allowed for (R1), (R2) and (R3).
}
\end{remarque}

\bt[\bf Invariance by refinement]\label{thm:refi}
Let $(X,\SS) \triangleleft (X,\mathcal T)$ be a refinement between two CS-sets. We suppose that there are no 1-exceptional strata. For any perversity $\ov q$ on $(X,\mathcal T)$ the identity 
$I \colon X \to X$ induces the isomorphisms

$$
\begin{array}{lclc}
\hbox{\rm (R1)}  &  \lau H {I^\star \ov  q}* {X;\SS} \cong \lau H {\ov q} * {X;\T} ,\phantom{--------}&
\hbox{\rm (R2)} & \lau H  * {I^\star \ov  q} {X;\SS} \cong \lau H * {\ov q}  {X;\T} ,
\\[,2cm]
\hbox{\rm (R3)}  &  \lau H * {I^\star \ov  q,c} {X;\SS} \cong \lau H * {\ov q,c}  {X;\T} ,\phantom{--------}&
\hbox{\rm (R4)}  &  \lau \gH {I^\star \ov  q} *{X;\SS} \cong \lau \gH { \ov q} * {X;\T},
\\[,2cm]
\hbox{\rm (R5)} & \lau  \gH * {I^\star \ov  q} {X;\SS} \cong \lau \gH * {\ov q}  {X;\T}   , \phantom{--------}&
\hbox{\rm (R6)}  &   \lau \gH * {I^\star \ov  q,c} {X;\SS} \cong \lau \gH * { \ov q,c}  {X;\T} .
\end{array}
$$
\noindent If in addition, $X$ is  second countable   then 
$$
\begin{array}{lclc}
\hbox{\rm (R7)}  & \lau H {BM,I^\star \ov  q}* {X;\SS} \cong \lau H {BM,  \ov q} * {X;\T}, 
 \phantom{---}&
\hbox{\rm (R8)} & \lau \gH {BM,I^\star \ov  q}* {X;\SS} \cong \lau \gH {BM,  \ov q} * {X;\T},
\\[,2cm]
\hbox{\rm (R9)}  &   \lau \IH * {I^\star \ov  q} {X;\T} \cong \lau \IH*  {\ov  q}  {X;\SS}, 
\phantom{---}&
\hbox{\rm (R10)}&  \lau  \IH * {I^\star\ov  q,c} {X;\SS} \cong \lau \H * { \ov q,c}  {X;\T}
\end{array}
$$
\et
\begin{proof}
 It suffices to apply  apply  \thmref{thm:Coars} to the perversity $I^\star \ov p$ (cf. Paragraph \ref{subsec:perv}), if this perversity is a K-perversity.
This is the case when 1-codimensional exceptional strata do not appear. Let us verify properties (K1) and (K2).

\smallskip

(K1) We have
$
I^\star \ov q (Q) = \ov q (\bi QI ) = \ov q (\bi SI) = I^\star \ov q(S) \leq  I^\star \ov q(Q) + \ov t (S) -\ov t(Q),
$
if we prove $\ov t(Q) \leq \ov t(S)$. This is clear if $S$ and $Q$ are regular strata or  singular strata at the same time (cf. (S4)). It remains the case where $S$ is an exceptional  stratum and $Q$ is a regular stratum. The inequality becomes $\ov t(S) \geq 0$, that is, $\codim S \geq 2$. This comes from the non-existence of 1-exceptional strata.

\smallskip

(K2) We have $I^\star \ov q(Q) = \ov q (\bi QI) = \ov q(\bi SI) = I^\star \ov q (S)$.
\end{proof}

In cases (R1), (R2) and (R3), 1-exceptional strata $S$ may appear if $\ov p(S)\geq 0$
(cf. \remref{rem:cod1excep}).

\subsection{Topological invariance}
One of the two more important properties of the intersection homology is the topological invariance \cite{MR572580}. 
Next Corollaries show that the refinement invariance implies topological invariance in some cases. We find the well known topological invariance of the intersection homology  \cite{MR572580} (see also \cite{MR800845,2019arXiv190406456F}) and those of tame intersection homology \cite{MR2209151} (closed supports) and  \cite{MR2507117} (compact supports).
We also get the topological invariance of the blown-up intersection cohomology 
\cite[Theorem G]{CST5} (closed  supports) and \cite[Theorem A]{CST7} (compact supports).

Before giving the result, there are two important tools  to highlight.

\smallskip

$\bullet$ \emph{Intrinsic stratification} (cf. \cite{MR0478169,MR800845}). Any stratified space $(X,\SS)$  has a smallest refinement: the 
\emph{intrinsic stratified space} $(X,\SS^*)$.
It is a canonical object:  we have $\SS^* ={\mathcal T}^*$ for any stratification $\mathcal T$ defined on $X$. 
If $(X,\SS)$ is a CS-set then $(X,\SS^*)$ is also a CS-set.

\smallskip

$\bullet$ \emph{Classical perversities versus M-perversities}. The former depend on the codimension of the strata while the latter are defined on the strata themselves.

A  {\em King perversity} is a map   $\ov p \colon \N \to \Z$ verifying $\ov p(0) =0 $ and $\ov p(k) \leq \ov p(k+1) \leq \ov p(k) +1 $ for each $k\in \N^*$ (cf. \cite{MR800845}). 
It verifies
\begin{equation}\label{eq:paso}
\ov p(k) \leq \ov p(\ell) \leq \ov p(k) +\ell -k ,
\end{equation}
if $1\leq k \leq \ell$.
 A King perversity $\ov p$ induces a perversity, still denoted by $\ov p$: $\ov p(S) = \ov p (\codim S)$.

A \emph{Goresky-MacPherson perversity} is a King perversity $\ov p$ with
 $
\ov p(0)=\ov p(1) =\ov p(2) = 0$ (cf. \cite{MR572580}).
It verifies, for each $k \geq 2$,
\begin{equation}\label{eq:pasot}
\ov 0 \leq \ov p(k) \leq k -2 = \ov t(k)
\end{equation}

\bc\label{Impl2}

Let  $(X,\SS)$ be a  CS-set endowed with a positive King perversity $\ov p$. 
Consider the intrinsic refinement $(X,\SS) \triangleleft_I (X,\SS^*)$. The identity map $I \colon X \to X$ induces the isomorphisms

\smallskip  

$\lau H {\ov p} *{X;\SS} \cong \lau H {\ov p} *{X;\SS^*}$
\hfill
$ \lau H *{\ov p} {X;\SS} \cong \lau H *{\ov p} {X;\SS^*}$
\hfill
$\lau H *{\ov p,c} {X;\SS} \cong  \lau H *{\ov p,c} {X;\SS^*}$,

\smallskip

\noindent if $\ov p (\ell) \geq 0$ when $\ell$ is the codimension of an exceptional stratum.
We also have

\smallskip

 $\lau \gH  {\ov p}* { X;\SS} \cong \lau \gH  {\ov p}* { X;\SS^*} $
 \hfill
 $\lau \gH * {\ov p} { X;\SS} \cong \lau \gH * {\ov p} { X;\SS^*}$
 \hfill
$  \lau \gH * {\ov p,c} { X;\SS} \cong \lau \gH * {\ov p,c} { X;\SS^*},$

  \smallskip
  
\noindent if  $0 \leq  \ov p(\ell) \leq  \ov t(\ell) $.
If in addition,  $X$ is second countable 
the we have

 \smallskip
 $
\lau \IH * {\ov p} {X;\SS} \cong \lau \IH * {\ov p} {X;\SS^*}$
\hfill
$\lau \IH  * {\ov p,c} {X;\SS} \cong \lau \IH  * {\ov p,c} {X;\SS^*}$
\hfill
$ \lau H {BM,\ov p} *{X;\SS} \cong \lau H {BM,\ov p} *{X;\SS^*}$
\ec

\bpr 
Let us verify that $\ov p$ is a $K$-perversity.

(K1) By definition of the perversity $\ov p$, we need to prove
\begin{equation*}\label{eq:desigprob}
\ov p(\codim Q)  \leq \ov p(\codim S)   \leq  \ov p(\codim Q)    + \ov t(\codim S) - \ov t(\codim Q).
\end{equation*} 
This is clear if $S, Q$ are regular strata or  singular strata (cf. (S4) and \eqref{eq:paso}). It remains the case where $S$ is an exceptional  stratum and $Q$ is a regular stratum. The inequality becomes $0  \leq \ov p(\codim S) \leq  \ov t(\codim S)$ which is true  from   hypothesis and \remref{rem:cod1excep}.

\smallskip

(K2) We have $\ov p(S) = \ov p  (\codim S) = \ov p(\codim Q) = \ov p (Q).
$

\medskip

The classical perversity $\ov p$ induces the perversity $\ov p$  on $(X,\SS)$ by formula $\ov p(S) = \ov p (\codim S)$.
In fact, the perversity $I_\star \ov p$ of $(X,\TT)$ also comes from the classical perversity $\ov p$:
$
I_\star \ov p(T ) \stackrel{\lemref{lem:I*p}} = \ov p(S) = \ov p (\codim S) \stackrel{source} = \ov p (\codim T) = \ov p(T),
$
where $T \in \TT$ and $S\in \SS$ is source stratum of $T$.
Now, it suffices to apply   \thmref{thm:Coars}. 
 \epr

\begin{remarque}\label{rem:class}
{\rm
(1) - Let  $(X,\SS)$ be a  CS-set endowed with a Goresky-MacPherson perversity $\ov p$.  Since $\ov p\geq \ov 0$ (cf. \eqref{eq:pasot}), then the previous Corollary implies that the cohomologies
$\lau H {\ov p} *{X;\SS}$, $\lau H *{\ov p} {X;\SS}$ and $\lau H *{\ov p,c} {X;\SS}$  
  are independent of the stratification $\SS$. We do not have a similar result for tame intersection homologies since condition $\ov 0 \leq \ov p\leq \ov t$ (cf. \eqref{eq:pasot}) implies that tame intersection homology coincides with the usual intersection homology.

Let us suppose that $X$ is second countable.
When 1-exceptional strata do not exist then we can apply the above Corollary 
and conclude that
the cohomologies
$\lau \IH * {\ov p} {X;\SS}$, $\lau \IH  * {\ov p,c} {X;\SS}$  and  $\lau H {BM,\ov p} *{X;\SS}$  
  are independent of the stratification $\SS$.
(cf. \eqref{eq:pasot}). 

(2) - Consider $\ov p$ a $K$-perversity. Condition (K2) means that the restriction of $\ov p$ to the $\SS$-stratification lying on each stratum $T \in \TT$ is in fact a classical perversity (excepted the condition $\ov p(0)= 0$). On the other hand, property (K1) is in fact a growing condition of the type \eqref{eq:paso}, even weaker. Although it is not completely exact, we can think a $K$-perversity as a perversity  whose restriction to any stratum $T \in \TT$ is a King perversity.
}\end{remarque}

\bibliographystyle{plain}

\bibliography{Biblio082017}

\end{document}